\documentclass[english]{amsart}
\usepackage[T1]{fontenc}
\usepackage[latin9]{inputenc}
\usepackage{geometry}
\usepackage{color}
\usepackage{babel}
\usepackage{units}
\usepackage{amsmath}
\usepackage{amsthm} 
\usepackage{amssymb}
\usepackage{setspace}
\usepackage{esint}
\usepackage[unicode=true,pdfusetitle,
 bookmarks=true,bookmarksnumbered=true,bookmarksopen=true,bookmarksopenlevel=1,
 breaklinks=false,pdfborder={0 0 1},backref=section,colorlinks=true]
 {hyperref}

\makeatletter
\numberwithin{equation}{section}
\numberwithin{figure}{section}
\theoremstyle{plain}
\newtheorem{thm}{Theorem}[section]
\newtheorem{theorem}{\protect\theoremname}[section]
\theoremstyle{definition}
\newtheorem{defn}[thm]{\protect\definitionname}
\theoremstyle{plain}
\newtheorem{cor}[thm]{\protect\corollaryname}
\theoremstyle{plain}
\newtheorem{lem}[thm]{\protect\lemmaname}
\theoremstyle{remark}
\newtheorem{rem}[thm]{\protect\remarkname}
\theoremstyle{plain}
\newtheorem{prop}[thm]{\protect\propositionname}

\makeatother

\providecommand{\corollaryname}{Corollary}
\providecommand{\definitionname}{Definition}
\providecommand{\lemmaname}{Lemma}
\providecommand{\propositionname}{Proposition}
\providecommand{\remarkname}{Remark}
\providecommand{\theoremname}{Theorem}
\global\long\def\C{\mathbb{C}}%
\global\long\def\R{\mathbb{R}}%
\global\long\def\N{\mathbb{N}}%
\global\long\def\Z{\mathbb{Z}}%
\global\long\def\F{\mathcal{F}}%

\global\long\def\n#1{\left\Vert #1\right\Vert }%

\DeclareMathOperator{\SL}{SL}
\DeclareMathOperator{\PSL}{PSL}
\DeclareMathOperator{\PGL}{PGL}

\begin{document}
\title{Cutoff on Graphs and the Sarnak-Xue Density of Eigenvalues}
\author{Konstantin Golubev~
and Amitay Kamber}
\thanks{Konstantin Golubev, k.golubev@gmail.com \\
D-MATH, ETH Zurich, Switzerland  \\
Amitay Kamber, ak2356@dpmms.cam.ac.uk \\ 
Centre for Mathematical Sciences, Wilberforce Road, Cambridge CB3 0WB, UK}
\begin{abstract}
It was recently shown in \cite{lubetzky2016cutoff} and \cite{sardari1diameter}
that Ramanujan graphs, i.e., graphs with the optimal spectrum, exhibit
cutoff of the simple random walk in an optimal time and have an optimal
almost-diameter. We show that this spectral condition can
be replaced by a weaker condition, the Sarnak-Xue density
property, to deduce similar results. This allows us to prove that some natural families of Schreier graphs of the
$\SL_{2}\left(\mathbb{F}_{t}\right)$-action on the projective line
exhibit cutoff, thus proving a special case of a conjecture of Rivin and Sardari, \cite{rivin2019quantum}. %To the best of our knowledge, this is the first example of an optimal almost-diameter and cutoff on an explicit family of graphs that are neither random nor Ramanujan.
\end{abstract}

\maketitle

\section{Introduction}

Various works (e.g.~\cite{bourgain2008uniform,lubotzky1988ramanujan})
proved that many families of $(q+1)$-regular graphs are expanders, and in particular their diameter (i.e., the largest distance between a pair of vertices) is equal to $\log_{q}\left(n\right)$
up to a multiplicative constant, where $n$ is the number of vertices.
In this paper, we show that under certain conditions, the distance
between most of the vertices in a graph is approximately optimal,
that is, for every $\epsilon>0$ and $n$ large enough, the distance is bounded by $(1+\epsilon)\log_{q}\left(n\right)$. We start with a special case of Schreier graphs of $\SL_{2}\left(\mathbb{F}_{t}\right)$,
which follows from Theorem~\ref{thm:random generators} below.
\begin{theorem}
\label{thm:Example}Let $\epsilon_0,\epsilon_1,\epsilon_2,\epsilon_3>0$ and $l\in\N$ be fixed. Then there exists $t_0$ such that the following holds. Let $t>t_0$ be a prime, let $\mathbb{F}_{t}$ be the finite field with $t$ elements, and let $P^{1}\left(\mathbb{F}_{t}\right)$ be the projective line over $\mathbb{F}_{t}$. Let $s_{1},s_{2},..,s_{l}$
be $l$ elements in $\SL_{2}\left(\mathbb{F}_{t}\right)$ chosen uniformly at random. Construct a $2l$-regular Schreier graph by connecting each point $\left[\begin{array}{c}
a\\
b
\end{array}\right]$ of $P^{1}\left(\mathbb{F}_{t}\right)$ to $s_{i}^{\pm}\left[\begin{array}{c}
a\\
b
\end{array}\right]$, $i=1,..,l$. Then with probability at least $1-\epsilon_0$
the following two statements hold:
\begin{itemize}
\item (Almost-Diameter) The distance between all but $\epsilon_1t^{2}$ of
the pairs $\left[\begin{array}{c}
a\\
b
\end{array}\right],\left[\begin{array}{c}
a'\\
b'
\end{array}\right]\in P^{1}\left(\mathbb{F}_{t}\right)$ satisfies 
\begin{align*}
d\left(\left[\begin{array}{c}
a\\
b
\end{array}\right],\left[\begin{array}{c}
a'\\
b'
\end{array}\right]\right)\le\left(1+\epsilon_2\right)\log_{2l-1}\left(t\right).
\end{align*}
\item (Cutoff) Consider the distribution of the simple random walk $A^{k}\delta_{x_{0}}$, 
starting from some $x_{0}=\left[\begin{array}{c}
a\\
b
\end{array}\right]\in P^{1}\left(\mathbb{F}_{t}\right)$. Here $A$ is the normalized adjacency operator on $L^{2}\left(P^{1}\left(\mathbb{F}_{t}\right)\right)$
defined by the graph structure, and $\delta_{x_{0}}$ is the probability
$\delta$-function supported on $x_{0}$. Then for all but $\epsilon_1 t$ of $x_{0}\in P^{1}\left(\mathbb{F}_{t}\right)$,
for every $k>\left(1+\epsilon_2\right)\frac{2l}{2l-2}\log_{2l-1}\left(n\right)$,
it holds that 
\begin{align*}
\n{A^{k}\delta_{x_{0}}-\pi}_{1}\le \epsilon_3,
\end{align*}
where $\pi$ is the constant
probability function on $P^1(\mathbb{F}_t)$.
\end{itemize}
\end{theorem}

Bourgain and Gamburd proved in \cite{bourgain2008uniform} that the
Cayley graphs of $\SL_{2}\left(\mathbb{F}_{t}\right)$ with respect
to random generators are expanders with probability tending to $1$
as $t\to\infty$. Since the graphs of Theorem~\ref{thm:Example}
are quotients of those graphs, they are expanders as well, which implies
that the distance between every two elements is bounded by $C\log_{2l-1}\left(t\right)$,
$C$ some constant. 
Theorem~\ref{thm:Example} further shows that for every $\epsilon>0$ and $t$ large enough, this constant is $(1+\epsilon)$ when we consider almost all the pairs. 

%Lubetzky and Peres proved in \cite{lubetzky2016cutoff} that the simple random walk on a Ramanujan graph exhibits cutoff for every starting vertex. The second part of the theorem above implies that the random walk on the graphs in Theorem~\ref{thm:Example} exhibits cutoff for almost every starting vertex (see Theorem~\ref{thm:Cutoff theorem} and its discussion).

Let us provide some context to this result. Let $\F$ be a family of finite $\left(q+1\right)$-regular connected graphs with the number of vertices tending to infinity. The graphs can have multiple edges and loops. Let $X\in\F$ be a graph from the family, and let $n$
denote its number of vertices. By $A\colon L^{2}\left(X\right)\to L^{2}\left(X\right)$
we denote the normalized adjacency operator of $X$
\begin{align*}
Af(x_{0})=\frac{1}{q+1}\sum_{x_{1}\sim x_{0}}f(x_{1}),
\end{align*}
where $x_1~\sim x_0$ means that $x_1$ is adjacent to $x_0$, and the sum is with multiplicity according to the number of edges between $x_0$ and $x_1$.

There are various results relating the eigenvalues of $A$ with the
geometry of $X$. In particular, if the largest eigenvalue of $A$
in absolute value, excluding $\pm1$, is bounded by some $\lambda<1$,
i.e., $X$ is an expander, then it is well known that the diameter of the graph is logarithmic in its size (\cite{hoory2006expander}).
Explicitly, in \cite{chung1994upper} it is proven that the diameter
is bounded by 
\begin{align*}
\left\lfloor \frac{\cosh^{-1}\left(n-1\right)}{\cosh^{-1}\left(1/\lambda\right)}\right\rfloor +1.
\end{align*}

A special case is when the graph is a Ramanujan graph, which means
that $\lambda\le2\sqrt{q}/\left(q+1\right)$. The Alon-Boppana theorem states that Ramanujan graphs are asymptotically the best spectral expanders. Then the upper bound on the diameter is $2\log_{q}\left(n\right)+O\left(1\right)$, as proved already by Lubotzky, Phillips and Sarnak (\cite{lubotzky1988ramanujan}).
Recently, it was proven by Lubetzky and Peres (\cite{lubetzky2016cutoff}) and independently by Sardari (\cite{sardari1diameter}) that when one considers the \emph{almost-diameter} of a Ramanujan graph $X$, i.e., the distance between most of the pairs of vertices in the graph, it is bounded by $(1+\epsilon)\log_{q}\left(n\right)$, which is, up to the $\epsilon$ factor, an optimal result. In \cite{lubetzky2016cutoff} it is also shown that the simple random walk on $X$ exhibits cutoff in the $L^{1}$-norm. The same results hold for a sequence of graphs that are almost-Ramanujan, i.e., for every $\epsilon>0$ and $n$ large enough, they satisfy $\lambda\le2\sqrt{q}/\left(q+1\right)+\epsilon$.

In practice, there are very few families of $(q+1)$-regular graphs that are known to be Ramanujan or almost-Ramanujan. For example, the question is completely open for the graphs underlying Theorem~\ref{thm:Example}, or other natural families of graphs we discuss below. It is widely believed that those graphs are almost-Ramanujan, but proving it seems to be currently out of reach. As a corollary, they should have an optimal almost-diameter and exhibit a cutoff of the simple random work. We refer to the work of Rivin and Sardari (\cite{rivin2019quantum}) for those conjectures and some experimental verification of them.

In this work, we show that some of the corollaries can be proved by replacing the Ramanujan condition with a weaker spectral \emph{density condition}. 
The density condition is based on the work of Sarnak and Xue (\cite{sarnak1991bounds})
on multiplicities of automorphic representations. More importantly, as Sarnak and Xue showed in the context of Lie groups of rank $1$, the density condition is also equivalent to a natural combinatorial path counting property, see Definition~\ref{def:weak-inj-radius}
and Theorem~\ref{thm:SX-equivalence}, allowing us to prove that it holds in many situations.

We use the following notations. For a real function $f(X,y,z)$ of the graph $X$ and auxiliary parameters $y,z$, we write $f(X,y,z)\ll_{\mathcal{F},y}g(X,y,z)$
if there exist constants $n_0,C$ depending on the family $\mathcal{F}$
and $y$, but not on $X$ and $z$, such that when the size $n$ of $X$ is greater than $n_0$, it holds that $f(X,y,z)\le C g(X,y,z)$.
The constants may also depend on the regularity $q$, but we do not mention it explicitly. The dependence on $\mathcal{F}$ is used to allow the family of graphs to depend on some other parameters, and we stress that the constant $C$ will be the same for all the graphs $X\in \mathcal{F}$.
We write $f(\cdot)=O_{\mathcal{F},y}\left(g\left(\cdot\right)\right)$
for $f\left(\cdot\right)\ll_{\mathcal{F},y}g\left(\cdot\right)$,
and $f(\cdot)\asymp_{\mathcal{F},y}g\left(\cdot\right)$ if both $f(\cdot)\ll_{\mathcal{F},y}g\left(\cdot\right)$
and $g(\cdot)\ll_{\mathcal{F},y}f\left(\cdot\right)$ take place.
We write $f(X,y,z)=o_{\mathcal{F},y}\left(g\left(X,y,z\right)\right)$,
if for every $c>0$ there exists $n_0$ depending on $\mathcal{F},y,c$ such that for $n=\left|X\right|>n_0$ it holds that $f(X,y,z)\le cg(X,y,z)$. 

Let $\lambda_{0}=1\ge\lambda_{1}\ge...\ge\lambda_{n-1}\ge-1$ be the
eigenvalues of $A$. We associate to $\lambda_{i}$ its \emph{$p$-value}
$p_{i}$ as follows 
\begin{equation*}
\begin{cases}
p_{i}=2 & \text{if }\left|\lambda_{i}\right|\le\frac{2}{q+1}\sqrt{q};\\
2<p_{i}\le\infty & \text{else, for }p_i\text{ such that } \left|\lambda_{i}\right|=\frac{1}{q+1}\left(q^{1/p_{i}}+q^{1-1/p_{i}}\right),
\end{cases}
\end{equation*}
where we use the convention that $1/\infty=\lim_{p\to\infty}1/p=0$.
In particular, $p_{0}=\infty$, and if $X$ is bipartite, then $\lambda_{n-1}=-1$
and $p_{n-1}=\infty$ as well. The definition of $p_{i}$ is based
on the action of $A$ on the $L^{p}$-functions on the $\left(q+1\right)$-regular
tree, namely, the $L^{p}$-norm of $A$ on the $\left(q+1\right)$-regular
tree is $\frac{1}{q+1}\left(q^{1/p}+q^{1-1/p}\right)$, see \cite{kamber2019p}.

\begin{defn}
We say that a family $\F$ of graphs satisfies \emph{the Sarnak-Xue
density property}, if for every graph $X\in\F$, $p>2$ and $\epsilon>0$,
\begin{align*}
\#\left\{ i:p_{i}\ge p\right\} \ll_{\F,\epsilon}n^{\nicefrac{2}{p}+\epsilon},
\end{align*}
where $n$ is the number of vertices of $X$.
\end{defn}

We prove two theorems that follow from the density property and will show below how one may prove the property. Theorem~\ref{thm:Cutoff theorem} was recently proved independently in \cite{bordenave2021cutoff}.

Since the number of vertices in a ball of radius $r$ in the $(q+1)$-regular tree is $1+\left(q+1\right)\frac{q^{r}-1}{q-1}\le3q^{r}$, the distance from a certain vertex to all but $o(n)$ of the other vertices is bounded from below by $\log_{q}\left(n\right)-o\left(1\right)$. This
shows that up to a $(1+o(1))$ factor, the following theorem is optimal.
\begin{theorem}
\label{thm:Almost-diameter theorem}If $\mathcal{F}$ is a family of expander graphs which satisfies the Sarnak-Xue density property, then for every $\epsilon_0, \epsilon_1, \epsilon_2>0$, there is $n_0$ depending on $\mathcal{F},\epsilon_0, \epsilon_1, \epsilon_2$, such that the following holds. For $X \in \mathcal{F}$ of size $|X|=n>n_0$, for
all but $\epsilon_0 n$ of $x_{0}\in X$,
all but $\epsilon_0 n$ of the vertices  of $X$ are within $\left(1+\epsilon_1\right)\log_{q}(n)$
distance from $x_{0}$, i.e.,
\begin{align*}
\#\left\{ y\in X\colon d\left(x_{0},y\right)>\left(1+\epsilon_1\right)\log_{q}(n)\right\} \le \epsilon_0 n.
\end{align*}
If, in addition, $X$ is vertex-transitive, then the
statement is true for all $x_{0}\in X$. Moreover, in such a case, $2\left(1+\epsilon_1\right)\log_{q}\left(n\right)$
is a bound on the diameter of $X$.
\end{theorem}

The second theorem concerns the cutoff phenomena, as discussed by
Lubetzky and Peres in \cite{lubetzky2016cutoff}. To simplify the
result, we assume that the graph $X$ is a non-bipartite graph. We
let $\delta_{x_{0}}\in L^{2}\left(X\right)$ be the delta probability
function supported on $x_{0}$, defined as $\delta_{x_{0}}\left(x_{0}\right)=1$
and $\delta_{x_{0}}\left(x\right)=0$ if $x\ne x_{0}$. Then $A^{k}\delta_{0}$
describes the probability distribution of the simple random walk that
starts at $x_{0}$ after $k$ steps. Since $X$ is non-bipartite,
this probability converges pointwise to the constant probability $\pi$,
defined as $\pi\left(x\right)=\frac{1}{n}$ for all $x\in X$. The
following theorem describes the speed of the convergence of this random
walk in the $L^{1}$-norm.
\begin{theorem}
\label{thm:Cutoff theorem}Assume that $X\in\mathcal{F}$ is a non-bipartite
graph. Then for every $\epsilon_0,\epsilon_1,\epsilon_2>0$:
\begin{enumerate}
\item There exists $n_0$ depending on $\epsilon_1,\epsilon_2$ such that for $X\in \mathcal{F}$, $|X|=n>n_0$, for every $x_{0}\in X$, for $k<\left(1-\epsilon_1\right)\frac{q+1}{q-1}\log_{q}\left(n\right)$,
it holds that 
\begin{align*} 
\n{A^{k}\delta_{x_{0}}-\pi}_{1}\ge2-\epsilon_2.
\end{align*} 
\item Assume moreover that $\mathcal{F}$ is a family of expander graphs which satisfies the Sarnak-Xue density property, then there exists $n_0$ depending on $\mathcal{F},\epsilon_0, \epsilon_1, \epsilon_2$ such that for $X\in \mathcal{F}$, $|X|=n>n_0$, for all but $\epsilon_0 n$
of $x_{0}\in X$, for every $k>\left(1+\epsilon_1\right)\frac{q+1}{q-1}\log_{q}\left(n\right)$,
it holds that 
\begin{align*}
\n{A^{k}\delta_{x_{0}}-\pi}_{1}\le \epsilon_2.
\end{align*}
If $X$ is moreover a vertex-transitive graph, the same result holds
for every $x_{0}\in X$.
\end{enumerate}
\end{theorem}

Collectively, we may summarize Theorem~\ref{thm:Almost-diameter theorem} and Theorem~\ref{thm:Cutoff theorem} by saying that \emph{an expander graph that satisfies the Sarnak-Xue density property has an optimal almost-diameter and exhibits cutoff}. We remark that Theorem~\ref{thm:Almost-diameter theorem} formally follows from Theorem~\ref{thm:Cutoff theorem}, but its proof is simpler and we think that it is of independent interest.

The cutoff phenomenon has been studied extensively in different settings in recent years and in particular regarding its relationship with the Ramanujan property (e.g.~\cite{lubetzky2016cutoff,lubetzky2020random,golubev2019cutoff,chapman2019cutoff}). The relation between the optimality of the almost-diameter and the Ramanujan property is also studied in different contexts. We already mentioned the work of Sardari (\cite{sardari1diameter}) and Lubetzky and Peres (\cite{lubetzky2016cutoff}), but it is also closely related to the work of Parzanchevski and Sarnak about Golden Gates (\cite{parzanchevski2018super}) and the general results of Ghosh, Gorodnik and Nevo on actions of semisimple groups (\cite{ghosh2014best}). 

The observation that the Ramanujan property can be replaced by the density condition was proved in the setting of hyperbolic surfaces in \cite{sarnak2015lettermiller,golubev2019cutoff}. Theorem~\ref{thm:Cutoff theorem},
which proves the analogous statement in the graph setting (and is actually easier), was recently 
proved independently in~\cite{bordenave2021cutoff}. 
%, but it seems like it made its first appearance, in a somewhat non-precise formulation, in a talk by Sarnak at the Israel Institute for Advanced studies in April, 2018 (its recording is available on the IIAS YouTube channel under the title ``Density Theorems and the Ramanujan Conjectures'').

The main novelty of our work is the \emph{verification of the density property} in a number of interesting cases, such as Theorem~\ref{thm:Example}. This is done by a path counting argument, which is explained in the rest of the introduction. 

\subsection*{The Weak Injectivity Radius Property}

We first give a parameterized version of the Sarnak-Xue density property
defined above.
\begin{defn}
\label{def:SX-density}We say that a family $\F$ satisfies \emph{the
Sarnak-Xue density property} \emph{with parameter }$0<A\le1$, if
for every $X\in\F$, $p>2$ and $\epsilon>0$, 
\begin{align*}
\#\left\{ i:p_{i}\ge p\right\} \ll_{\epsilon,\F}n^{1-A\left(1-\nicefrac{2}{p}\right)+\epsilon}.
\end{align*}
\end{defn}

Note that the trivial eigenvalue $1$ (and also $-1$ if $X$ is bipartite)
is always counted in the left hand side of the inequality, so the parameter $A$ is always bounded from above by $1$. Ramanujan graphs automatically
satisfy the density property with parameter $A=1$, since for $p>2$,
$\#\left\{ i:p_{i}\ge p\right\} \in\left\{ 1,2\right\} $ (depending
on whether the graph is bipartite or not). The density property does
not necessarily imply uniform expansion of the family, since arbitrary
large eigenvalues can appear as long as there are few of them. However,
if $X$ is a Cayley graph of a \emph{quasirandom group} $G$ in the
sense of \cite{gowers2008quasirandom}, meaning that the smallest
non-trivial representation of $G$ is of dimension $\gg\left|G\right|^{\beta}$,
then there is a lower bound on the multiplicity of every eigenvalue
$\lambda_{i}$. Then if $A+\beta>1$, it holds that for $p>\frac{2A}{\beta+A-1}$
and $n$ large enough, $\#\left\{ i:p_{i}\ge p\right\} =1$, implying
expansion. This was indeed Sarnak and Xue's approach to the proof of spectral
gap for congruence subgroups of arithmetic cocompact subgroups of
$\SL_{2}\left(\R\right)$ and $\SL_{2}\left(\C\right)$. We remark that the idea was actually used earlier in a similar context by Huxley (\cite{huxley1986exceptional}). Notice that the group $\SL_2(\mathbb{F}_t)$ is quasirandom with parameter $\beta=1/3$, which implies that $A> 2/3$ guarantees expansion. 

Another interesting result is the connection between Sarnak-Xue density and Benjamini-Schramm convergence. It follows from the results of Abert, Glasner and Virag (\cite{abert2016measurable})
that the Sarnak-Xue density property with any parameter $A>0$ implies
Benjamini-Schramm convergence of the family to the $(q+1)$-regular
tree (see Section~\ref{sec:Benjamini-Schramm}).

We are finally ready to state the geometric definition underlying our work.
\begin{defn}
\label{def:weak-inj-radius}Let $X$ be a $(q+1)$-regular graph and $x_{0}\in X$
be a vertex. Let $P\left(X,k,x_{0}\right)$ be the number of non-backtracking
paths of length $k$ starting and ending at $x_{0}$, and let $P\left(X,k\right)=\sum_{x_{0}\in X}P\left(X,k,x_{0}\right)$.
We say that $X$ satisfies the \emph{Sarnak-Xue weak injective radius
property with parameter} $0<A\le1$, if for  $k=2\lfloor A\log_{q}n \rfloor$,
we have for every $\epsilon>0$ 
\begin{equation}
P\left(X,k\right)\ll_{\epsilon,\mathcal{F}}n^{1+\epsilon}q^{k/2}\asymp n^{1+A+\epsilon}.\label{eq:Pathcount}
\end{equation}
\end{defn}
For the parameter $A=1$, the property is equivalent to the fact that
\begin{equation}\label{eq:Weak injective Radius}
    P(X,2\lfloor \log_{q}n \rfloor)\ll_{\epsilon,\mathcal{F}}n^{2+\epsilon}.
\end{equation}

Sarnak and Xue (\cite{sarnak1991bounds}) essentially proved the
if part of the following theorem in the context of Lie groups of rank
$1$. They later showed that the weak injective radius property, and
hence the density property, holds in all sequences of principal congruence
subgroups of cocompact arithmetic subgroups of $\SL_{2}\left(\R\right)$
and $\SL_{2}\left(\C\right)$.
\begin{theorem}
\label{thm:SX-equivalence}The Sarnak-Xue density property with parameter
$0<A\le1$ is satisfied if and only if the Sarnak-Xue weak injective
radius property with the same parameter $A$ is satisfied.

Therefore, if the weak injective
radius property is satisfied with parameter $A=1$ and the graphs of the family are expanders, then they have optimal almost-diameter and exhibit cutoff (as stated in Theorem~\ref{thm:Almost-diameter theorem} and Theorem~\ref{thm:Cutoff theorem}).
\end{theorem}

The girth of a graph, i.e., the length of the shortest cycle, is
equal to twice the injectivity radius of the graph (plus one, if the girth is odd). Therefore, if there are no cycles in $X$ of length
$\le2A\log_{q}n$, or equivalently the graph has injective radius
$\ge A\log_{q}n$, then the weak injective radius property with parameter $A$ is automatically satisfied. However, in many cases one can show that the weak injective radius property is satisfied for larger values of $A$. 
For example, Ramanujan graphs satisfy the weak injective radius property with parameter $A=1$ (as follows from Theorem~\ref{thm:SX-equivalence}),
while it is not known if Ramanujan graphs with girth close to $2\log_{q}\left(n\right)$
exist. The best known asymptotic result about girth is that the bipartite Ramanujan graphs constructed by Lubotzky, Phillips and Sarnak have girth at least $4/3\log_{q}\left(n\right)-O(1)$. Another example is that by \cite{gamburd2009girth}, random Cayley graphs in $\SL_{2}\left(\mathbb{F}_{t}\right)$ have girth at least $\left(1/3-o\left(1\right)\right)\log_{q}\left(n\right)$, while we show that the weak injective radius property holds with parameter $1/3$, that is essentially double of what follows from the girth.

\subsection*{Graphs Satisfying Sarnak-Xue Density Property}

We focus on Sarnak-Xue density property with parameter $A=1$. There
are a number of simple examples of $\left(q+1\right)$-regular graphs satisfying this property.

\paragraph*{Ramanujan graphs.}

This is straightforward- by definition, there are no eigenvalues with $p$-value greater than $2$ except for the trivial ones. Let us remark that while the proof that the LPS graphs of \cite{lubotzky1988ramanujan} (or similar number theoretic graphs) are Ramanujan is far from being elementary and eventually relies heavily on the machinery of algebraic geometry, the fact that the Sarnak-Xue density
property holds for them requires only elementary number theoretical
tools, and is contained implicitly in \cite[Theorem 4.4.4]{davidoff2003elementary}. Therefore, our work provides an elementary proof for the cutoff phenomena for those graphs. See also Theorem~\ref{thm:LPS} and its proof for this argument.

\paragraph*{One-sided Ramanujan graphs.}

These are graphs with $\lambda_{i}\le2\sqrt{q}/(q+1)$ for every $i>0$.
Such non-bipartite graphs are constructed by the interlacing polynomials
method of \cite{marcus2013interlacing} and \cite{hall2018ramanujan}.
See Proposition~\ref{prop:Positive eigenvalues only} for the proof
of this case.

\paragraph*{Random regular graphs.}

With high probability a random $\left(q+1\right)$-regular graph of
size $n$ and, more generally, with high probability a random $n$-cover
(or $n$-lift) of a fixed $\left(q+1\right)$-regular graph is almost-Ramanujan, and therefore satisfies
the Sarnak-Xue density property. This follows from Alon's conjecture, which was proved by Friedman (\cite{friedman2003proof}),
and its extension to random lifts proved by Bordenave (\cite{bordenave2020new}). As a matter of fact, using Theorem~\ref{thm:SX-equivalence} this actually follows from the much simpler work of Broder and Shamir \cite{broder1987second}.

\subsection*{Schreier Graphs of $\SL_{2}\left(\mathbb{F}_{t}\right)$}

Let $t$ be prime and $P^{1}\left(\mathbb{F}_{t}\right)$ be the projective
line over $\mathbb{F}_{t}$, with an action of $\SL_{2}\left(\mathbb{F}_{t}\right)$
by Mobius transformations. Since the action is transitive, we may consider $P^{1}(\mathbb{F}_{t})$ as a quotient of $\SL_{2}\left(\mathbb{F}_{t}\right)$. Let $S_{t}\subset \SL_{2}\left(\mathbb{F}_{t}\right)$
be a symmetric generating set, and let $X_{t}=\text{Cayley}\left(\SL_{2}\left(\mathbb{F}_{t}\right),S_{t}\right)$
be the corresponding Cayley graph. Let $Y_{t}$ be the Schreier quotient
of $X_{t}$, whose vertex set is $P^{1}\left(\mathbb{F}_{t}\right)$. 

Expansion and the Ramanujan condition on $X_{t}$ and $Y_{t}$ for
various choices and generators was studied in various works. The two
most important ones are \cite{lubotzky1988ramanujan} which found Ramanujan examples using deep results in number theory, and \cite{bourgain2008uniform}
which proved that such Cayley graphs with logarithmic girth are expanders, but the resulting bounds are quite far from the Ramanujan bound. 
In \cite{rivin2019quantum}, Rivin and Sardari conjecture and verify experimentally that for many choices of generators both graphs should be almost-Ramanujan, and should therefore have an optimal almost-diameter and exhibit cutoff.

We note that, unlike expansion, the fact that density holds for a graph $X$ does not immediately imply that it holds for some quotient $Y$ of it. However, in this special case, we have the following \emph{density amplification theorem}: 
\begin{theorem}
\label{thm:Projective space theorem}If $X_{t}$ satisfies the Sarnak-Xue
density property with parameter $A\ge1/3$, then $Y_{t}$ satisfies
this property with parameter $A=1$. In particular, if the graph $Y_{t}$ is also an expander, then it has an optimal almost-diameter
and exhibits cutoff (as stated in Theorem~\ref{thm:Almost-diameter theorem} and Theorem~\ref{thm:Cutoff theorem}).
\end{theorem}

One can also replace $\SL_{2}\left(\mathbb{F}_{t}\right)$ in Theorem~\ref{thm:Projective space theorem}
by either $\PGL_{2}\left(\mathbb{F}_{t}\right)$ or $\PSL_{2}\left(\mathbb{F}_{t}\right)$.

We remark that Theorem~\ref{thm:Projective space theorem} is not based on quasirandomness, but on a simple counting argument. Indeed, quasirandomness implies that if we assume that the graph $X_{t}$ satisfies the Sarnak-Xue density property with parameter $A$, then for the graph $Y_{t}$ it holds that
\begin{align*}
\#\left\{ i:p_{i}\ge p\right\} \ll_{\epsilon,\F}n^{2-3A\left(1-\nicefrac{2}{p}\right)+\epsilon},
\end{align*}
which is a different estimate, and in particular has non-trivial implications only for $A\ge2/3$.

Here are three interesting cases for which we can prove that $X_{t}$
satisfies the Sarnak-Xue density property with parameter $A\ge1/3$.
In all of these cases, one can also show expansion for $X_{t}$, and
hence for $Y_{t}$, using the results of Bourgain and Gamburd (\cite{bourgain2008uniform}).
Recall that we assume that $t$ is prime.
\begin{theorem}
\label{thm:SL2(Z)}Let $S\subset \SL_{2}\left(\Z\right)$ be a symmetric
set of size $\left|S\right|=\left(q+1\right)$, which generates a
free subgroup in $\SL_{2}\left(\Z\right)$, and assume that each element
$s\in S$ has operator norm $\n s\le q$. Then the graphs $X_{t}=\text{Cayley}\left(\SL_{2}\left(\mathbb{F}_{t}\right),S\:\mod t\right)$
satisfy the Sarnak-Xue density property with parameter $A=1/3$.

Therefore, by Theorem~\ref{thm:Projective space theorem}, the graphs $Y_t$ have an optimal almost-diameter and exhibit cutoff.
\end{theorem}

A particularly interesting case is the set
\begin{align*}
S=\left\{ \left(\begin{array}{cc}
1 & \pm2\\
0 & 1
\end{array}\right),\left(\begin{array}{cc}
1 & 0\\
\pm2 & 1
\end{array}\right)\right\} .
\end{align*}
If $\pm2$ is replaced with $\pm1$ then the subgroup generated in
$\SL_{2}$ is not free, so the corresponding graphs $X_{t}$ do not
even Benjamini-Schramm converge to the $4$-regular infinite tree
and, in particular, do not satisfy the density property. On the other
hand, for $S=\left\{ \left(\begin{array}{cc}
1 & \pm3\\
0 & 1
\end{array}\right),\left(\begin{array}{cc}
1 & 0\\
\pm3 & 1
\end{array}\right)\right\} $ the norm of every element is greater than $q=3$, so it is not known
whether the graphs satisfy the density property with parameter $A\geq1/3$.
We conjecture that they do, and moreover, in Theorem~\ref{thm:SL2(Z)},
the condition on the operator norm bound is redundant.
\begin{theorem}
\label{thm:LPS}For primes $q,t>2$ let $X^{q,t}$ be the $\left(q+1\right)$-regular
Cayley graphs of $\PSL_{2}\left(\mathbb{F}_{t}\right)$ or $\PSL_{2}\left(\mathbb{F}_{t}\right)$
constructed by Davidoff, Sarnak and Valette in \cite{davidoff2003elementary}.
Let $S_{t}$ be a symmetric subset of the generators of $X^{q,t}$
of size $\left|S_{t}\right|=q'+1\ge\sqrt{q}+1$, and let $X_{t}$
be the $\left(q'+1\right)$-regular Cayley graph generated by $S_{t}$.
Then the graphs $X_{t}$ satisfy the Sarnak-Xue density property with
parameter $A=1/3$. Therefore, the graphs $Y_t$ have an optimal almost-diameter and exhibit cutoff.

\end{theorem}

The graphs $X^{q,t}$ above are usually denoted by $X^{p,q}$ but we
denote differently to avoid confusion with the rest of the article. These graphs are a slight generalization of the LPS graphs of \cite{lubotzky1988ramanujan} since the congruence conditions for $p$ and $q$ is not assumed. 

The following theorem, together with Theorem~\ref{thm:Projective space theorem},
implies Theorem~\ref{thm:Example}.
\begin{theorem}
\label{thm:random generators}Let $S_{t}$ be a random set of size $\left(q+1\right)/2$ in $\SL_{2}\left(\mathbb{F}_{t}\right)$, and
$X_{t}=\text{Cayley}\left(\SL_{2}\left(\mathbb{F}_{t}\right),S_{t}\cup S_{t}^{-1}\right)$.
Then for every $\epsilon_0>0$, with probability $1-\epsilon_0$, the graphs $X_{t}$ satisfy the Sarnak-Xue density property with parameter
$A=1/3$. Therefore, the graphs $Y_t$ have an optimal almost-diameter and exhibit cutoff.
\end{theorem}

As the probabilistic statement in the theorem is somewhat confusing, let us state it in another form. The theorem says that for every $\epsilon_0>0$, among all $M_t$ of the possible graphs $X_t$, we may choose $(1-\epsilon_0)M_t$ of them, and the resulting family of graphs will satisfy the Sarnak-Xue density property with parameter $A=1/3$. The underlying constants in Definition~\ref{def:SX-density} will depend on $\epsilon_0$.

The proof follows from the results of \cite{gamburd2009girth}. It
is a natural question to ask whether density holds for $X_{t}$ with
parameter $A=1$. It is notable that the methods that are used to
prove that random graphs and random covers of graphs are almost Ramanujan
(\cite{friedman2003proof,bordenave2021cutoff}) do not work in this case.

\subsection*{Structure of the Paper}

We start with the application of the results to the Schreier graphs on the projective line in Section~\ref{sec:projective space}, where
we prove Theorem~\ref{thm:Projective space theorem}, Theorem~\ref{thm:SL2(Z)}, Theorem~\ref{thm:LPS} and Theorem~\ref{thm:random generators}.

In Section~\ref{sec:Preliminaries-from-Spectral} we recall some basic results from the spectral theory of graphs. In Section~\ref{sec:The-Density-Condition}
we prove Theorem~\ref{thm:SX-equivalence} and discuss other conditions that are equivalent to the density property. 

In Section~\ref{sec:Ramanujan-Graphs} we prove Theorem~\ref{thm:Almost-diameter theorem}
and Theorem~\ref{thm:Cutoff theorem} for Ramanujan graphs. These results also appear in \cite{lubetzky2016cutoff,sardari1diameter}, but we
provide the proof for them in order to simplify the proof of the next results. In Section~\ref{sec:Proof-of-almost-density} we prove Theorem~\ref{thm:Almost-diameter theorem} and Theorem~\ref{thm:Cutoff theorem}.
Similar results recently appeared independently in \cite{bordenave2021cutoff}, and are included here for self-containment and since our proofs are slightly different.

Finally, in Section~\ref{sec:Benjamini-Schramm} we shortly discuss the connection between the density property and Benjamini-Schramm convergence.

\subsection*{Acknowledgments}

The authors are thankful to Michael Chapman, Ori Parzanchevski and Peter Sarnak for fruitful discussions, and to the referees for their constructive remarks and suggestions. During the work on this project, the first author was supported by the SNF grant 200020-169106 at ETH Zurich. This work was part of the PhD thesis of the second author at the Hebrew University of Jerusalem, under the guidance of Prof. Alex Lubotzky, and was supported by the ERC grant 692854.

\section{\label{sec:projective space}Schreier Graphs on the Projective Line}

Let $t$ be a prime, $\mathbb{F}_{t}$ the finite field with $t$ elements,  $G_{t}=\SL_{2}\left(\mathbb{F}_{t}\right)$, and
$P_{t}=P^{1}\left(\mathbb{F}_{t}\right)$ be the projective line over
$\mathbb{F}_{t}$, i.e., the set of vectors $\left[\begin{array}{c}
a\\
b
\end{array}\right]$ such that $a,b\in\mathbb{F}_{t}$, not both zero, quotiented by the
equivalence relation $\left[\begin{array}{c}
a\\
b
\end{array}\right]\sim\left[\begin{array}{c}
a'\\
b'
\end{array}\right]$ if and only if $ba'=ab'$.

The group $G_{t}$ acts transitively on $P^{1}\left(\mathbb{F}_{t}\right)$,
the stabilizer of $\left[\begin{array}{c}
1\\
0
\end{array}\right]\in P^{1}\left(\mathbb{F}_{t}\right)$ is 
\begin{align*}
K_{t}=\left\{ \left(\begin{array}{cc}
a & b\\
0 & a^{-1}
\end{array}\right);a\in\mathbb{F}_{t}^{\times},b\in\mathbb{F}_{t}\right\},
\end{align*}and therefore $P_{t}$ can be identified with the  space
$P_{t}\simeq G_{t}/K_{t}$. Explicitly, $\left[\begin{array}{c}
a\\
b
\end{array}\right]\in P^{1}\left(\mathbb{F}_{t}\right)$ corresponds to the set of matrices $A\in \SL_{2}\left(\mathbb{F}_{t}\right)$,
with first column equivalent to $\left[\begin{array}{c}
a\\
b
\end{array}\right]$.

Let $S_{t}\subset G_{t}$ be a symmetric generating set of $G_{t}$,
with $\left|S_{t}\right|=q+1$. Let $X_{t}$ the Cayley graph of $G_{t}$
w.r.t. $S_{t}$, and let $Y_{t}$ be the Schreier graph of the quotient
$G_{t}/K_{t}$ w.r.t. $S_{t}$.
\begin{theorem}
If $X_{t}$ has weak injective radius with parameter $A$, then $Y_{t}$
has weak injective radius with parameter $\min\left\{ 1,3A\right\} $.
\end{theorem}

\begin{proof}
For simplicity, we assume that $S_{t}=R_{t}\cup R_{t}^{-1}$ with $\left|R_{t}\right|=\left|S_{t}\right|/2$.
We will also first assume that $G_{t}=\PSL_{2}\left(\mathbb{F}_{t}\right)$.
Let $F_{R}$ be the free group generated by a set $R$ of size $\left|R\right|=\left|R_{t}\right|=\left(q+1\right)/2$, and let $l\colon F_R \to \N_{\ge 0}$ be the standard length, i.e. $l(g)$ is the minimal length of a product of generators and their inverses which is equal to $g$.
An identification $R\simeq R_{t}$ defines a homomorphism $\varphi_{t}\colon F_{R}\to G_{t}$.

Note that if $N\subset F_{R}$ is a finite index subgroup which defines
a Schreier graph $X$ on the vertex set $F_R/N$, then the weak injective radius property
with parameter $A$ holds for $X$ if and only if for $k\le2A\log_{q}\left(\left|X\right|\right)$ and every $\epsilon>0$,
\begin{align*}
\#\left\{ \left(g,xN\right)\in F_{R}\times F_R / N:l\left(g\right)=k,\,gxN=xN\right\} \ll_{\epsilon}\left|X\right|^{1+\epsilon}q^{k/2},
\end{align*}

Denote $N_{t}=\ker\varphi_{t}$ and $M_{t}=\left\{ g\in F_{R}:\varphi_{t}\left(g\right)\left[\begin{array}{c}
1\\
0
\end{array}\right]=\left[\begin{array}{c}
1\\
0
\end{array}\right]\right\} $, where the equality is in $P^{1}\left(\mathbb{F}_{t}\right)$ (alternatively
one can say that $\left(\begin{array}{c}
1\\
0
\end{array}\right)\in\mathbb{F}_{t}^{2}$ is an eigenvector of $\varphi_{t}\left(g\right)$). $M_{t}$ is a
subgroup of $F_{R}$ and contains $N_{t}$ which is its normal core
in $F_{R}$. $M_{t}$ defines the graph $Y_{t}$, while $N_{t}$ defines
the graph $X_{t}$. Since $N_{t}$ is normal in $F_{R}$, and since
$\left|X_{t}\right|=\left(t-1\right)\left(t+1\right)t/2\asymp t^{3}$,
the assumed Sarnak-Xue weak injective property with parameter $A$
for $X_{t}$ implies that for $k\le2\left\lfloor A\log_{q}\left(\left|X_{t}\right|\right)\right\rfloor =6A\log_{q}\left(t\right)+O(1)$ and every $\epsilon>0$, it holds that 
\begin{align*}
\#\left\{ g\in N_{t}:l\left(g\right)=k\right\} \ll_{\epsilon}t^{\epsilon}q^{k/2}.
\end{align*}
Let us bound the size of the set $M_{t,k}=\left\{ \left(g,y\right)\in F_{R}\times P^{1}\left(\mathbb{F}_{t}\right):l\left(g\right)=k,\,\varphi\left(g\right)y=y\right\} $.
If $\left(g,y\right)\in M_{t,k}$, then $y\in P^{1}\left(\mathbb{F}_{t}\right)$
is a \emph{projective eigenvector} of $\varphi_{t}\left(g\right)$,
i.e., the lift of $y$ to $\mathbb{F}_{t}^{2}$ is an eigenvector
of $\varphi_{t}\left(g\right)$. There are two options: either $\varphi_{t}\left(g\right)=I\in G_{t}$
(i.e., $g\in N_{t}$) and then $\varphi_{t}\left(g\right)$ has $(t+1)$
projective eigenvectors, or $\varphi_{t}\left(g\right)\ne I\in G_{t}$,
and $\varphi_{t}\left(g\right)$ has at most 2 projective eigenvectors
(note that here the assumption that the group is actually $\PSL_{2}\left(\mathbb{F}_{t}\right)$
is used). Hence, for $k\le6A\log_{q}\left(t\right)-O\left(1\right)=2\left(3A\right)\log_{q}\left(t\right)-O\left(1\right)$,
\begin{align*}
\left|M_{t,k}\right| & \le\left(t+1\right)\left|\left\{ g\in N_{t}:l\left(g\right)=k\right\} \right|+2\left|\left\{ g\in F_{R}:l\left(g\right)=k\right\} \right|\\
 & \ll_{\epsilon}\left|Y_{t}\right|^{1+\epsilon}q^{k/2}+q^{k}.
\end{align*}
Note that for $k\le2\log_{q}\left(\left|Y_{t}\right|\right)$ it holds
that 
\begin{align*}
q^{k}\le\left|Y_{t}\right|q^{k/2},
\end{align*}
and this implies that the weak injective property for $Y_{t}$ holds
with parameter $\min\left\{ 1,3A\right\} $.

The treatment of the case $\SL_{2}$ is similar to that of $\PSL_{2}$
except for the case of the elements $g\in F_{R}$ such that $\varphi_{t}\left(g\right)=-I$.
Such elements also have $(t+1)$ projective eigenvectors. We treat this case
by showing that the number of such elements of length $k\le6A\log_{q}\left(t\right)$
is also bounded by $\ll_{\epsilon}t^{\epsilon}q^{k/2}$. This follows
from the general Lemma~\ref{lem:Cayley paths} in Section~\ref{sec:Proof-of-almost-density}.
\end{proof}

We now prove that the three examples from the introduction have Sarnak-Xue weak injective radius $A=1/3$.
\begin{proof}[Proof of Theorem~\ref{thm:SL2(Z)}]
Let $S\subset \SL_{2}\left(\Z\right)$ be a symmetric set of size
$\left|S\right|=\left(q+1\right)$, which generates a free group in
$\SL_{2}\left(\Z\right)$, and assume that each element $s\in S$ has operator norm $\n s\le q$. We need to show that $X_{t}=\text{Cayley}\left(\SL_{2}\left(\mathbb{F}_{t}\right),S\mod t\right)$
has weak injective radius with parameter $A=1/3$.

Denote $\Gamma=\SL_{2}\left(\Z\right)$, and let $\Gamma(t) = \{\gamma\in \Gamma : \gamma\equiv I \mod t \}$ be its principal congruence subgroup of level $t$. Let $B_{T}=\left\{ A\in M_{n}\left(\R\right):\n A\le T\right\}$.

Then it is proved in (\cite{gamburd2002spectral}) that
\begin{equation}
|\Gamma\left(t\right)\cap B_{T}|\ll_{\epsilon}\left(T/t^{2}+1\right)\left(T/t+1\right)T^{\epsilon}.\label{eq:lattice point count}
\end{equation}

For completeness, let us give a short proof of this fact. Write 
\begin{align*}
\gamma=\left(\begin{array}{cc}
a & b\\
c & d
\end{array}\right).
\end{align*}
It is sufficient to prove Equation~\eqref{eq:lattice point count} for the norm $\n{\gamma}_{\infty}=\max\left\{ \left|a\right|,\left|b\right|,\left|c\right|,\left|d\right|\right\} $, since all the norms on $M_{2}\left(\R\right)$ are equivalent. Then if $\gamma\in\Gamma\left(t\right)$, i.e.~$\gamma=I\mod t$, it is a simple exercise to see that 
\begin{align*}
a+d=2\mod t^{2}.
\end{align*}
Then if $\n{\gamma}_{\infty}\le T$ there are $\ll(T/t^{2}+1)$
options for $a+d$, there are $\ll(T/t+1)$ options for $a$, and therefore $\ll\left(T/t^{2}+1\right)\left(T/t+1\right)$ options for $a,d$.
Then if $ad\ne1$ we have $bc=1-ad\ne0$ and by standard bounds on the divisor function (\cite[Theorem~315]{hardy1979introduction}) there are $\ll_{\epsilon}T^{\epsilon}$
options for $b,c$. 
When $ad=1$ then there are $2$ options for $a,d$, and $bc=0$, so there are $\ll(T/t+1)$ options for $b,c$.

Returning to the proof, by Equation~\eqref{eq:lattice point count}, the number of paths of length $k$ in $X_{t}$ is bounded by 
\begin{align*}
|\Gamma\left(t\right)\cap B_{q^{k}}|\ll_{\epsilon}\left(q^{k}/t^{2}+1\right)\left(q^{k}/t+1\right)q^{k\epsilon}.
\end{align*}
For $k\le2A\log_{q}\left(t^{3}\right)=2\log_{q}\left(t\right)$, i.e.~$t\ge q^{k/2}$
it holds that 
\begin{align*}
\left(q^{k}/t^{2}+1\right)\left(q^{k}/t+1\right)\ll q^{k/2},
\end{align*}
which means that $X_{t}$ satisfies the Sarnak-Xue density property
with parameter $A=1/3$.
\end{proof}
Let us complete the analysis of 
\begin{align*}
S=\left\{ \left(\begin{array}{cc}
1 & \pm2\\
0 & 1
\end{array}\right),\left(\begin{array}{cc}
1 & 0\\
\pm2 & 1
\end{array}\right)\right\} .
\end{align*}
We need to calculate the norm of the generators. Denote by $\lambda_{\max}$
the maximal eigenvalue of a semi-definite matrix. Then
\begin{align*}
\n{\left(\begin{array}{cc}
1 & 2\\
0 & 1
\end{array}\right)}^{2} & =\lambda_{max}\left(\left(\begin{array}{cc}
1 & 2\\
0 & 1
\end{array}\right)\left(\begin{array}{cc}
1 & 0\\
2 & 1
\end{array}\right)\right)=\lambda_{max}\left(\begin{array}{cc}
5 & 2\\
2 & 1
\end{array}\right)\\
 & =\frac{6\pm\sqrt{36-1}}{2}\le6\le3^{2}=q^{2},
\end{align*}
 and similarly for the other generators. 
\begin{proof}[Proof of Theorem~\ref{thm:LPS}]
For primes $q,t>2$ let $X^{q,t}$ be the $\left(q+1\right)$-regular
Cayley graphs of $\PSL_{2}\left(\mathbb{F}_{t}\right)$ or $\PSL_{2}\left(\mathbb{F}_{t}\right)$
constructed by Davidoff, Sarnak and Valette in \cite{davidoff2003elementary}.
Let $S_{t}$ be a symmetric subset of the generators of $X^{q,t}$
of size $\left|S_{t}\right|= q'+1\ge\sqrt{q}+1$, and let $X_{t}$
be the $\left(q'+1\right)$-regular Cayley graph generated by $S_{t}$. 

Then as shown in \cite[Lemma 4.4.2]{davidoff2003elementary}, for
$k=0\mod2$, the number of cycles $P\left(X^{p,t},k,\text{id}\right)$
is bounded by the number $s_{Q}\left(q^{k}\right)$ of ways to represent
$q^{k}$ by the quadratic form 
\begin{align*}
Q\left(x_{0},x_{1},x_{2},x_{3}\right)=x_{0}^{2}+4t^{2}\left(x_{1}^{2}+x_{2}^{2}+x_{3}^{2}\right).
\end{align*}
Following the analysis in \cite[Theorem 4.4.4]{davidoff2003elementary} which we will not repeat but is quite elementary,
\begin{align*}
    s_Q(q^k)\ll_\epsilon q^{k\epsilon}
    \left(q^k/t^3+q^{k/2}/t \right)
\end{align*}
Therefore, 
\begin{align*}
P\left(X_{t},k,\text{id}\right) & \le P\left(X^{p,t},k,\text{id}\right)\ll_{\epsilon}q^{k\epsilon}\left(q^k/t^3+q^{k/2}/t\right)\\
 & \le q^{\prime2k\epsilon}\left(q^{\prime 2k}/t^{3}+q^{\prime k}/t\right)
\end{align*}
For $k\le2\log_{q'}t$ it therefore holds that 
\begin{align*}
P\left(X_{t},k,\text{id}\right)\ll_{\epsilon}t^{\epsilon}q^{\prime k/2},
\end{align*}
as needed.
\end{proof}
\begin{proof}[Proof of Theorem~\ref{thm:random generators}]
Let $S_{t}$ be a random set of size $\left(q+1\right)/2$ in $\SL_{2}\left(\mathbb{F}_{t}\right)$,
and $X_{t}=\text{Cayley}\left(\SL_{2}\left(\mathbb{F}_{t}\right),S_{t}\cup S_{t}^{-1}\right)$.
By \cite[Lemma 10]{gamburd2009girth}, the probability that a word
$w$ of length $k$, when $k=O\left(t/\log\left(t\right)\right)$,
will evaluate to $1$ is $\le\frac{k}{t}+o\left(t^{-2}\right)$. Let
$k\le2\log_{q}\left(t\right)$. Then the expected number of cycles
of length $k$ is bounded by:
\begin{align*}
E\left(P\left(X_{t},k\right)\right)\ll q^{k}\frac{k}{t}\le kq^{k/2}.
\end{align*}

Let $\epsilon_0>0,\epsilon>0$ and let $t_0=t_0(\epsilon)$ be large enough so that for $t>t_0$ it holds that $\log_q(t)\le t^\epsilon$. Then, by Markov's inequality, the probability for $t>t_0$ that $P\left(X_{t},k\right)>C_{\epsilon}q^{k/2}t^{\epsilon}$, is
bounded by $CC_{\epsilon}^{-1}t^{-\epsilon}k\le2CC_{\epsilon}^{-1}\log_{q}\left(t\right)t^{-\epsilon} \ll C_\epsilon$.
By choosing $C_{\epsilon}=C_1 \epsilon_0 m^{2}$ for $\epsilon=\frac{1}{m}$, we can
ensure that for every $t$, with probability $(1-\epsilon_0)$, for
every $\epsilon>0$ of the form $\epsilon=\frac{1}{m}$, if $t>t_0(\epsilon)$ it will hold that $P\left(X_{t},k\right)\le C_{\epsilon}q^{k/2}t^{\epsilon}$.
Therefore, the same will hold for every $\epsilon>0$, as needed.
\end{proof}

\section{\label{sec:Preliminaries-from-Spectral}Preliminaries on the Spectral
Theory of Graphs}

We start with some basics of the spectral theory of graphs. See \cite{kamber2019p}
for some more details.

As before, let $X$ be a finite $(q+1)$-regular graph with $n$ vertices,
possibly with multiple edges and self-loops. A \emph{non-backtracking path} is a path in the graph such that no two consecutive steps take the \emph{same} edge in the opposite
direction. For $x_0\in X$, there are $(q+1)q^{k-1}$ non-backtracking paths starting at $x_0$, and we write $T_k(x_0,x_1)$ as the number of non-backtracking paths of length $k$ that start at $x_0$ and finish at $x_1$.

For $k\ge0$, define
the distance $k$ Hecke-operator $A_{k}\colon L^{2}\left(X\right)\to L^{2}\left(X\right)$
by 
\begin{align*}
A_{k}f\left(x_{0}\right)=\frac{1}{\left(q+1\right)q^{k-1}}\sum_{x_1 \in X}T_k(x_0,x_1)f\left(x_1\right).
\end{align*}
 On the $(q+1)$-regular infinite tree $A_{k}$ acts
by averaging over the sphere of radius $k$ around a vertex.

Note that the path-counting functions defined in Definition \ref{def:weak-inj-radius}
can be expressed as
\begin{align*}
P\left(X,k,x_{0}\right) & =\left(q+1\right)q^{k-1}\left\langle A_{k}\delta_{x_{0}},\delta_{x_{0}}\right\rangle \\
P\left(X,k\right) & =\left(q+1\right)q^{k-1}\text{tr}A_{k}.
\end{align*}

It holds that $A_{0}=Id,A_{1}=A$, and the following recursive relation
takes place for $k\ge1$

\begin{align*}
AA_{k}=\frac{q}{q+1}A_{k+1}+\frac{1}{q+1}A_{k-1}.
\end{align*}

Therefore, if $v\in L^{2}\left(X\right)$ is an eigenfunction of $A$,
i.e., $Av=\lambda v$, then $A_{k}v=\lambda^{\left(k\right)}v$, where
$\lambda^{(k)}$ is a function of $\lambda$ and $k$. In order to
calculate $\lambda^{(k)}$, write 
\begin{align*}
\lambda=\frac{1}{q+1}\left(\theta+q\theta^{-1}\right),
\end{align*}
for some $0\ne\theta\in\C$. This equation always has two (possibly,
equal) solutions $\theta_{\pm}$ satisfying $\theta_{+}\theta_{-}=q$,
and
\begin{align*}
\theta_{\pm}=\frac{\left(q+1\right)\lambda\pm\sqrt{\left(q+1\right)\lambda^{2}-4q^{2}}}{2}.
\end{align*}
Solving this equation and recalling that $A$ is a self-adjoint operator
of norm $1$, so its eigenvalues are real and satisfy $-1\le\lambda\le1$,
we have:
\begin{itemize}
\item If $\left|\lambda\right|\le\frac{2\sqrt{q}}{q+1}$, then $\left|\theta_{\pm}\right|=\sqrt{q}$.
In this case, we let $p\left(\lambda\right)=2$ be the $p$-value
of $\lambda$, and we let $\theta$ be one of $\theta_{\pm}$.
\item If $\frac{2\sqrt{q}}{q+1}<\left|\lambda\right|\leq1$, then both $\theta_{\pm}$
are real, and we let $\theta$ be the larger one in absolute value.
It holds that $\sqrt{q}<\left|\theta\right|\le q$ and $\theta$ has
the same sign as $\lambda$. It also holds that $\left|\theta\right|=q^{1-1/p}$,
where $2<p=p(\lambda)\le\infty$ is the $p$-value of $\lambda$.
\end{itemize}
The relation between $\theta$ and the eigenvalues $\lambda^{\left(k\right)}$
of $A_{k}$ is given by the following formula, which may be verified
by induction 
\begin{align}
\lambda^{\left(k\right)} & =\frac{1}{\left(q+1\right)q^{k-1}}\left(\theta^{k}+\left(q\theta^{-1}\right)^{k}+\left(1-q^{-1}\right)\sum_{i=1}^{k-1}q^{i}\theta^{k-2i}\right).\label{eq: explicit expression}
\end{align}
The following proposition provides an upper bound and a lower bound for $\lambda^{(k)}$.
\begin{cor}
\label{cor:eigenvalue calculations}If $p\left(\lambda\right)=2$
then for every $\epsilon>0$, 
\begin{align*}
\left|\lambda^{\left(k\right)}\right|\le\left(k+1\right)q^{-k/2}\ll_{\epsilon}q^{k\left(-1/2+\epsilon\right)}.
\end{align*}
If $p\left(\lambda\right)>2$ then for every $\epsilon>0$,
\begin{align*}
q^{-k/p\left(\lambda\right)}\le\left|\lambda^{\left(k\right)}\right|\le\left(k+1\right)q^{-k/p\left(\lambda\right)}\ll_{\epsilon}q^{k\left(-1/p\left(\lambda\right)+\epsilon\right)}.
\end{align*}
Moreover, if $k=0\mod2$ then $\lambda^{\left(k\right)}$ is positive,
otherwise $\lambda^{\left(k\right)}$ has the same sign as $\lambda$.
\end{cor}

\begin{proof}
The upper bounds follow directly from the explicit expression in Equation~\eqref{eq: explicit expression}.
Note that since $\left|\theta\right|\ge\left|q\theta^{-1}\right|$,
the value $\theta^{k}$ is the largest one in absolute value in Equation~\eqref{eq: explicit expression}. 

It is left to prove the lower bound in the case of $p=p\left(\lambda\right)>2$.
We can assume that $\lambda>0$, so $\theta>0$, and hence all the
summands in Equation~\eqref{eq: explicit expression} are positive.
Consider the function $f\colon\R\to\R$ defined as $f\left(x\right)=q^{-x/p}$.
This function is decreasing, hence for any $x_{1},...,x_{N}\in\R$
\begin{align*}
\frac{1}{N}\sum_{i=1}^{N}f\left(x_{i}\right)\ge f\left(\max\left\{ x_{i}\right\} \right).
\end{align*}
Apply this to the following multiset (i.e., a set with repetitions)
of numbers: $q^{k}$ times $k$, $\left(q-1\right)q^{k-i-1}$ times
$k-2i$ for $0<i<k$, and once the value $-k$. There are $N$ elements
in this multiset, where
\begin{align*}
N=q^{k}+\left(q-1\right)\sum_{i=1}^{k-1}q^{k-1-i}+1=q^{k}+q^{k-1}=\left(q+1\right)q^{k-1}.
\end{align*}
Therefore, we have 
\begin{align*}
\frac{1}{\left(q+1\right)q^{k-1}}\left(q^{k}q^{-k/p}+\left(q-1\right)\sum_{i=1}^{k-1}q^{k-1-i}q^{-\left(k-2i\right)/p}+1\cdot q^{k/p}\right)\ge f\left(k\right)=q^{-k/p}.
\end{align*}
The left hand side is equal to $\lambda^{\left(k\right)}$.
\end{proof}

\section{\label{sec:The-Density-Condition}Equivalence of the Density and
the Path-Counting Properties}

Recall that the Sarnak-Xue density property with parameter $A$ for
a graph $X$ with $n$ vertices states that for every $p>2$ and $\epsilon>0$,
\begin{align*}
\#\left\{ i:p_{i}\ge p\right\} \ll_{\epsilon,\F}n^{1-A\left(1-\nicefrac{2}{p}\right)+\epsilon}.
\end{align*}
Here $p_{i}$ is the $p$-value of the $i$-th eigenvalue.
\begin{lem}
\label{lem:partial sum lemma}The Sarnak-Xue density property with parameter $A$ is equivalent to the fact that for every $\epsilon>0$,
\begin{equation}
\sum_{i}n^{-1+A\left(1-\nicefrac{2}{p_{i}}\right)}\ll_{\epsilon,\mathcal{F}}n^{\epsilon}.\label{eq: partial sum condition}
\end{equation}
\end{lem}

\begin{proof}
Assume Inequality~\eqref{eq: partial sum condition} holds. Then for
every $\epsilon>0$ and $p>2$,
\begin{align*}
n^{\epsilon} & \gg_{\epsilon,\mathcal{F}}\sum_{i:p_{i}\ge p}n^{-1+A\left(1-\nicefrac{2}{p_{i}}\right)}\\
 & \ge\sum_{i:p_{i}\ge p}n^{-1+A\left(1-\nicefrac{2}{p}\right)}\\
 & =n^{-1+A\left(1-\nicefrac{2}{p}\right)}\#\left\{ i:p_{i}\ge p\right\} ,
\end{align*}
which implies the Sarnak-Xue density property with parameter $A$.

For the other direction, we apply discrete integration by parts (see
\cite[Theorem 421]{hardy1979introduction}). For a smooth function
$f\colon\left[2,\infty\right)\to\R$ and a finite sequence of points $x_{i}$,
let $M\left(x\right)=\#\left\{ i:x_{i}\ge x\right\} $. Then the following
holds
\begin{align*}
\sum_{i}f\left(x_{i}\right)=M\left(2\right)f\left(2\right)+\intop_{2}^{\infty}M\left(x\right)\frac{\partial}{\partial x}f\left(x\right)dx.
\end{align*}
Let $f\left(x\right)=n^{-1+A\left(1-2/x\right)}$ and assume the Sarnak-Xue
density property with parameter $A$, then
\begin{align*}
\sum_{i}n^{-1+A\left(1-\nicefrac{2}{p_{i}}\right)} & =\sum_{i}f\left(p_{i}\right)=n^{-1}\#\left\{ i:p_{i}\ge2\right\} +\intop_{2}^{\infty}\#\left\{ i:p_{i}\ge x\right\} \frac{\partial}{\partial x}n^{-1+A\left(1-2/x\right)}dx\\
 & \ll_{\epsilon,\mathcal{F}}n^{-1}\cdot n+\intop_{2}^{\infty}n^{1-A\left(1-2/x\right)+\epsilon}\ln\left(n\right)A\frac{2}{x^{2}}n^{-1+A\left(1-2/x\right)}dx\\
 & \ll 1+n^{2\epsilon}\intop_{2}^{\infty}\frac{1}{x^{2}}dx\ll n^{2\epsilon}.
\end{align*}
\end{proof}
We can now give a proof of Theorem~\ref{thm:SX-equivalence}.
\begin{proof}[Proof of Theorem~\ref{thm:SX-equivalence}]
First assume that Sarnak-Xue density property with parameter $A$
is satisfied for $X$. We shall prove that $P\left(X,k\right)\ll_{\epsilon,\F} n^{1+\epsilon}q^{k/2}$
for $k=2\lfloor A\log_{q}\left(n\right)\rfloor$. Indeed:
\begin{align*}
P\left(X,k\right) & =\left(q+1\right)q^{k-1}\text{tr}A_{k}=\\
 & =\left(q+1\right)q^{k-1}\sum_{i=1}^{n}\lambda_{i}^{\left(k\right)}\\
 & \le\left(q+1\right)q^{k-1}\sum_{i=1}^{n}\left|\lambda_{i}^{\left(k\right)}\right|\\
 & \le\left(q+1\right)q^{k-1}\sum_{i=1}^{n}\left(q+1\right)q^{-k/p_{i}}\\
 & \ll\sum_{i=1}^{n}q^{k\left(1-1/p_{i}\right)}=\sum_{i=1}^{n}n^{2A\left(1-1/p_{i}\right)}\\
 & =n^{A}n\sum_{i=1}^{n}n^{-1+A\left(1-1/p_{i}\right)}\\
 & \ll_{\epsilon,\F}q^{k/2}nn^{\epsilon}=n^{1+\epsilon}q^{k/2}.
\end{align*}
Here we used the upper bounds of Corollary~\ref{cor:eigenvalue calculations}
in the fourth line, and Lemma~\ref{lem:partial sum lemma} in the
last line. 

For the other direction, assume the injective radius property with
parameter $A$. Let $k=2\left\lfloor A\log_{q}n\right\rfloor $. Then
by the weak injective radius property we have for every $\epsilon>0$,
\begin{align*}
n^{1+A+\epsilon} & \asymp n^{1+\epsilon}q^{k/2}\gg_{\epsilon,\F}P\left(X,k\right)=\\
 & =\left(q+1\right)q^{k-1}\sum_{i=1}^{n}\lambda_{i}^{\left(k\right)}\\
 & \gg q^{k}\sum_{i:p_{i}>2}\lambda_{i}^{\left(k\right)}+q^{k}\sum_{i:p_{i}=2}\lambda_{i}^{\left(k\right)}\\
 & \ge q^{k}\sum_{i:p_{i}>2}q^{-k/p_{i}}+q^{k}\sum_{i:p_{i}=2}\lambda_{i}^{\left(k\right)}\\
 & \ge q^{k}\sum_{i:p_{i}>2}n^{-2A/p_{i}}+q^{k}\sum_{i:p_{i}=2}\lambda_{i}^{\left(k\right)}.
\end{align*}
Here we applied the lower bound of Corollary~\ref{cor:eigenvalue calculations}
in the last line and used the fact that $k$ is even. By re-arranging
and applying the upper bounds of Corollary~\ref{cor:eigenvalue calculations},
we have 
\begin{align*}
\sum_{i:p_{i}>2}n^{-2A/p_{i}} & \ll_{\epsilon,\F}q^{-k}n^{1+A+\epsilon}-\sum_{i:p_{i}=2}\lambda_{i}^{\left(k\right)}\\
 & \ll_{\epsilon}n^{1-A+\epsilon}+nq^{-k\left(1/2-\epsilon\right)}\\
 & \asymp n^{1-A+\epsilon}.
\end{align*}
This proves that the Sarnak-Xue density property holds as in Lemma~\ref{lem:partial sum lemma}.
\end{proof}

\begin{rem}
The proof theorem implies that if the weak injective radius property holds for the parameter $A$, then it holds for any parameter $A'\le A$, which is not obvious from the definition.
\end{rem}

One may also weaken the Sarnak-Xue density property, so we have to consider only positive eigenvalues of $A$. We have:

\begin{prop}
\label{prop:Positive eigenvalues only}The Sarnak-Xue weak injective radius property \emph{with parameter $A$} ($0<A\le1$) holds if for every
$p>2$ and $\epsilon>0$, 
\begin{align*}
\#\left\{ i:\lambda_{i}>0,p_{i}\ge p\right\} \ll_{\epsilon,\F}n^{1-A\left(1-\nicefrac{2}{p}\right)+\epsilon}.
\end{align*}
\end{prop}

Note that the difference between this property and the Sarnak-Xue density property (Definition \ref{def:SX-density}) is that we restricted ourselves to positive eigenvalues. One cannot hope to prove the analogous statement for the negative eigenvalues only, as the smallest negative eigenvalue may stay within the Ramanujan range even without Benjamini-Schramm convergence to the $\left(q+1\right)$-regular tree (\cite{cameron1976line}).
\begin{proof}
Assume the weaker condition. Let $k=2\left\lfloor A\log_{q}n\right\rfloor +1$.
From the geometric interpretation of $\text{tr}A_{k}$ as path counting,
we know that $\text{tr}A_{k}\ge0$. Therefore, 
\begin{align*}
\sum_{i:p_{i}=2}\lambda_{i}^{\left(k\right)}+\sum_{i:p_{i}>2,\lambda_{i}>0}\lambda_{i}^{\left(k\right)}+\sum_{i:p_{i}>2,\lambda_{i}<0}\lambda_{i}^{\left(k\right)}\ge0.
\end{align*}
Noticing that $k$ is odd, and using Corollary~\ref{cor:eigenvalue calculations} and Lemma~\ref{lem:partial sum lemma}
we get for every $\epsilon>0$,
\begin{align*}
\sum_{i:p_{i}>2,\lambda_{i}<0}n^{-2A/p_{i}} & \asymp\sum_{i:p_{i}>2,\lambda_{i}<0}q^{-k/p_{i}}\\
 & \le-\sum_{i:p_{i}>2,\lambda_{i}<0}\lambda_{i}^{\left(k\right)}\\
 & \le\sum_{i:p_{i}>2,\lambda_{i}>0}\lambda_{i}^{\left(k\right)}+\sum_{i:p_{i}=2}\lambda_{i}^{\left(k\right)}\\
 & \ll_{\epsilon}\sum_{i:p_{i}>2,\lambda_{i}>0}q^{-k\left(1/p_{i}-\epsilon\right)}+\sum_{i:p_{i}=2}q^{-k\left(1/2-\epsilon\right)}\\
 & \ll\sum_{i:p_{i}>2,\lambda_{i}>0}n^{-2A/p_{i}+\epsilon}+n\cdot n^{-A\left(1+\epsilon\right)}\\
 & \ll_{\epsilon}n^{1-A+\epsilon}.
\end{align*}
Applying Lemma~\ref{lem:partial sum lemma} again we have for every
$p>2$ and $\epsilon>0$,
\begin{align*}
\#\left\{ i:\lambda_{i}<0,p_{i}\ge p\right\} \ll_{\epsilon,\F}n^{1-A\left(1-\nicefrac{2}{p}\right)+\epsilon},
\end{align*}
and together with the weaker condition we have the full Sarnak-Xue density property.
\end{proof}

Let us discuss non-backtracking cycles, i.e. non-backtracking paths of length $k\ge2$ such that the last edge is not the inverse of the first edge. We consider two cycles
as equivalent if they are rotations of one another. We say that a cycle $C$ is primitive if there is no $l|k$ and a cycle $C'$ of length $l$, such that $C$ is equivalent to the concatenation of
$C'$ to itself $k/l$ times. 

The Sarnak-Xue density property with parameter $A$ is also equivalent to a bound on the number of primitive cycles of length $k\le2A\log_{q}\left(n\right)$.
The theorem below shows a stronger result, that it is enough to consider only cycles of length $k=2\left\lfloor A\log_{q}\left(n\right)\right\rfloor $. 

The theorem is based on the \emph{graph prime number theorem} of \cite{terras2010zeta}.
\begin{theorem} 
Denote by $\pi_{X}\left(k\right)$ the number of equivalence classes of primitive cycles of length $k$, and by $N_{X}\left(k\right)$ the number of cycles of length $k$ (i.e. proper cycles, not equivalence classes).

For a sequence of graphs $X\in\mathcal{F}$, the Sarnak-Xue density property with parameter $A$ is satisfied if and only if for $k=2\left\lfloor A\log_{q}\left(n\right)\right\rfloor $ and $\epsilon>0$,
either $\pi_{X}\left(k\right)\ll_{\epsilon,\F}n^{1+\epsilon}q^{k/2}$
or $N_{X}\left(k\right)\ll_{\epsilon,\F}n^{1+\epsilon}q^{k/2}$.
\end{theorem}

\begin{proof}
We first show that the last two conditions are equivalent. We have (see \cite[Section 10]{terras2010zeta})
\begin{align*}
N_{X}\left(k\right)=\sum_{m|k}m\pi_{X}\left(m\right),
\end{align*}
which implies 
\begin{align*}
\pi_{X}\left(k\right)=\frac{1}{k}\sum_{m|k}\mu\left(\frac{k}{m}\right)N_{X}\left(m\right),
\end{align*}
where $\mu$ is the Mobius function. Then 
\begin{align*}
\left|\pi_{X}\left(k\right)-\frac{1}{k}N_{X}\left(k\right)\right|=\frac{1}{k}\sum_{m|k,m\ne k}\mu\left(\frac{k}{m}\right)N_{X}\left(m\right).
\end{align*}
Using the trivial bound 
\begin{align*}
N_{X}\left(m\right)\ll nq^{m},
\end{align*}
the fact that the largest divisor of $k$ is $k/2$, and $\sum_{m|k}1\le k$,
we have 
\begin{align*}
\left|\pi_{X}\left(k\right)-\frac{1}{k}N_{X}\left(k\right)\right|\ll_{\epsilon}nq^{k/2}.
\end{align*}
Using $k\ll_{\epsilon}n^{\epsilon}$, we see that $\pi_{X}\left(k\right)\ll_{\epsilon,\F}n^{1+\epsilon}q^{k/2}$
if and only if $N_{X}\left(k\right)\ll_{\epsilon,\F}n^{1+\epsilon}q^{k/2}$.

We now show the claim for $N_{X}\left(k\right)$. Let $E_{X}$ be
the set of directed edges of $X$, for $e\in E_{X}$ let $o(e),t(e)\in X$
be the origin and terminus of $e$, and let $\bar{e}\in E_{X}$ be
the opposite edge. Let $H\colon L^{2}\left(E_{X}\right)\to L^{2}\left(E_{X}\right)$
be Hashimoto's non-backtracking operator, defined by 
\begin{align*}
Hf\left(e\right)=\sum_{e':o(e')=t(e),e'\ne\bar{e}}f\left(e'\right).
\end{align*}
The following is well-known, and follows from the theory of the graph
Ihara zeta function, see \cite[Section 10]{terras2010zeta} :
\begin{enumerate}
\item $N_{X}\left(k\right)=\text{tr}H^{k}$.
\item If the eigenvalues of $A$ are $\lambda_{1},..,\lambda_{n}$, then
the eigenvalues of $H$ are $\theta_{i,\pm}$, defined as in Section~\ref{sec:Preliminaries-from-Spectral},
by 
\begin{align*}
\theta_{i,\pm}=\frac{\left(q+1\right)\lambda_{i}\pm\sqrt{\left(q+1\right)\lambda_{i}^{2}-4q^{2}}}{2},
\end{align*}
and also has the eigenvalues $\pm1$, each with multiplicity $\left|E_{X}\right|/2-n$.
\end{enumerate}
Therefore, 
\begin{align*}
N_{X}\left(k\right)=\left(1+\left(-1\right)^{k}\right)\left(\left|E_{X}\right|/2-n\right)+\sum_{i=1}^{n}\left(\theta_{i,+}^{k}+\theta_{i,-}^{k}\right).
\end{align*}
Note that for $k=0\mod2$ and $p_{i}>2$ we have 
\begin{align*}
\theta_{i,+}^{k}+\theta_{i,-}^{k}=q^{k\left(1-1/p_{i}\right)}+q^{k/p_{i}}.
\end{align*}
Choose $k=2\left\lfloor A\log_{q}\left(n\right)\right\rfloor $. We
then have 
\begin{align*}
N_{X}\left(k\right)-\sum_{i:p_{i}>2}\left(\theta_{i,+}^{k}+\theta_{i,-}^{k}\right) & =N_{X}\left(k\right)-\sum_{i:p_{i}>2}\left(q^{k\left(1-1/p_{i}\right)}+q^{k/p_{i}}\right)=\\
 & =2\left(\left|E_{X}\right|/2-n\right)+\sum_{i:p_{i}=2}\left(\theta_{i,+}^{k}+\theta_{i,-}^{k}\right).
\end{align*}
Using the fact that for $p=2$ it holds that $\left|\theta\right|=\sqrt{q}$
we have 
\begin{align*}
\left|N_{X}\left(k\right)-\sum_{i:p_{i}>2}q^{k\left(1-1/p_{i}\right)}\right|\ll nq^{k/2}.
\end{align*}
Therefore, $N_{X}\left(k\right)\ll_{\epsilon,\F}n^{1+\epsilon}q^{k/2}$
if and only if $\sum_{i:p_{i}>2}q^{k\left(1-1/p_{i}\right)}\ll_{\epsilon,\F}n^{1+\epsilon}q^{k/2}$.
Using the fact that $k=2\left\lfloor A\log_{q}\left(n\right)\right\rfloor $
and Lemma \ref{lem:partial sum lemma} we get the claim for $N_{X}\left(k\right)$. 
\end{proof}

\section{\label{sec:Ramanujan-Graphs}Ramanujan Graphs}

In this section, we discuss the analogues of Theorems \ref{thm:Almost-diameter theorem}
and \ref{thm:Cutoff theorem} for Ramanujan graphs. The results can
essentially be found in the work of Lubetzky and Peres (\cite{lubetzky2016cutoff})
or Sardari (\cite{sardari1diameter}). We include them here in a slightly
different form in order to make use of them later.

In what follows in this section, we assume the graphs to be non-bipartite.
However, the bipartite case can be treated similarly with the main difference that the eigenvalue $\left(-1\right)$ has to be taken into account just as $1$, as we explain now. If $X$ is non-bipartite,
define  
\begin{align*}
L_{0}^{2}\left(X\right)=\left\{ f\in L^{2}\left(X\right):\sum_{x\in X}f\left(x\right)=0\right\} .
\end{align*}
If $X$ is an expander then the norm of $A$ on $L_{0}^{2}\left(X\right)$
is bounded by $\lambda_{0}<1$. Let $\pi\left(x\right)=\frac{1}{n}$
be the constant probability function. Then we may write every delta
function as 
\begin{equation}
\delta_{x_{0}}=\pi+\left(\delta_{x_{0}}-\pi\right),\label{eq:Decomposition}
\end{equation}
with $\delta_{x_{0}}-\pi\in L_{0}^{2}\left(X\right)$. It holds that
$A_{k}\pi=\pi$ for every $k\ge0$.

If $X$ is bipartite, the vertex set can be decomposed into two equal
parts $X=X_{L}\cup X_{R}$, where $\left|X_{L}\right|=\left|X_{R}\right|=n/2$,
and all the edges are between a vertex in $X_{L}$ and a vertex in
$X_{R}$. Then define 
\begin{align*}
L_{00}^{2}=\left\{ f\in L^{2}\left(X\right):\sum_{x\in X_{L}}f\left(x\right)=\sum_{x\in X_{R}}f\left(x\right)=0\right\} .
\end{align*}
In this case, if $X$ is an expander then the norm of $A$ on $L_{00}^{2}\left(X\right)$
is bounded by $\lambda_{0}<1$. Let $\pi_{-}\in L^{2}\left(X\right)$
be the function 
\begin{align*}
\pi_{-}\left(x\right)=\begin{cases}
\frac{1}{n} & x\in X_{L}\\
-\frac{1}{n} & x\in X_{R}.
\end{cases}
\end{align*}
Then the delta function for $x_{0}\in X_{L}$ can be written as 
\begin{equation}
\delta_{x_{0}}=\pi+\pi_{-}+\left(\delta_{x_{0}}-\pi-\pi_{-}\right)\label{eq:Decomposition Bipartite}
\end{equation}
with $\delta_{x_{0}}-\pi-\pi_{-}\in L_{00}^{2}\left(X\right)$, and
analogously for $x_{0}\in X_{R}$ it holds that 
\begin{equation}
\delta_{x_{0}}=\pi-\pi_{-}+\left(\delta_{x_{0}}-\pi+\pi_{-}\right),\label{eq:Decomposition Bipartite 2}
\end{equation}
with $\delta_{x_{0}}-\pi+\pi_{-}\in L_{00}^{2}\left(X\right)$. It
holds that $A_{k}\pi=\pi$ for every $k\ge0$ and $A^{k}\pi_{-}=\left(-1\right)^{k}\pi_{-}$.

Using Equations~\eqref{eq:Decomposition}, \eqref{eq:Decomposition Bipartite}
and \eqref{eq:Decomposition Bipartite 2} we can understand the action
of $A_{k}$, for $k\ge0$, on $\delta_{x_{0}}$.

For simplicity, we consider from now on the non-bipartite case only.
The statements and proofs can be extended to the bipartite case with
minor adjustments.
\begin{lem}
\label{lem:L2 to diameter lemma}Let $X$ be a $\left(q+1\right)$-regular
non-bipartite graph of size $n$. Assume that for $b\in\R$, $c>0$,
\begin{align*}
\n{A_{k}\left(\delta_{x_{0}}-\pi\right)}_{2}^{2}\le cn^{-b}.
\end{align*}
Then 
\begin{align*}
\#\left\{ y:d\left(x_{0},y\right)>k\right\} \le cn^{-b+2}.
\end{align*}
\end{lem}

\begin{proof}
The result follows from the fact that if $d\left(x_{0},y\right)>k$,
then $A_{k}\left(\delta_{x_{0}}-\pi\right)\left(y\right)=-\frac{1}{n}$
and therefore,
\begin{align*}
\n{A_{k}\left(\delta_{x_{0}}-\pi\right)}_{2}^{2}\ge\#\left\{ y:d\left(x_{0},y\right)>k\right\} n^{-2}.
\end{align*}
\end{proof}
\begin{theorem}
\label{thm:Diameter-Ramanujan}If $X$ is a Ramanujan graph, then for every $\tau>0$, 
\begin{align*}
\#\left\{ y\in X\colon d\left(x_{0},y\right)>\left(1+\tau \right)\log_{q}(n)\right\} \le ((1+\tau)\log_q(n)+1)^2n^{1-\tau} = o_\tau\left(n\right),
\end{align*}
and for every $\epsilon>0$, for $n$ large enough, $\left(2+\epsilon\right)\log_{q}\left(n\right)$
is an upper bound on the diameter of $X$.
\end{theorem}

\begin{proof}
For simplicity assume that $X$ is non-bipartite, the bipartite case is similar. 

Write $\delta_{x_{0}}$ as 
\begin{align*}
\delta_{x_{0}}=\pi+\left(\delta_{x_{0}}-\pi\right),
\end{align*}
with $\delta_{x_{0}}-\pi\in L_{0}^{2}\left(X\right)$. Then, by Corollary~\ref{cor:eigenvalue calculations}
\begin{align*}
\n{A_{k}\left(\delta_{x_{0}}-\pi\right)}_{2}\le\left(k+1\right)q^{-k/2}\n{\delta_{x_{0}}}_{2}=\left(k+1\right)q^{-k/2}.
\end{align*}
And for $k=\left(1+\tau\right)\log_{q}\left(n\right)$, 
\begin{align*}
\n{A_{k}\left(\delta_{x_{0}}-\pi\right)}_{2}^{2}\le ((1+\tau)\log_q(n)+1)^2n^{-1-\tau}
\end{align*}
which implies, by Lemma~\ref{lem:L2 to diameter lemma}, that 
\begin{align*}
\#\left\{ y:d\left(x_{0},y\right)>\left(1+\epsilon\right)\log_{q}\left(n\right)\right\} \le ((1+\tau)\log_q(n)+1)^2n^{1-\tau}.
\end{align*}

The diameter bound can be deduced from the almost-diameter bound,
since for $n$ large enough every two vertices have a third vertex of distance at most $\left(1+\epsilon/2\right)\log_{q}\left(n\right)$
from the both of them. However, it can also be deduced directly by
taking $\tau = 1+\epsilon$. 
\end{proof}
Let us give some remarks on the proof and some generalizations. 
\begin{enumerate}
\item The same proof shows that if every non-trivial eigenvalue $\lambda$
of $X$ is bounded by $\left|\lambda\right|\le\lambda_{0}=\frac{1}{q+1}\left(q^{1/p}+q^{1-1/p}\right)$,
then the almost-diameter is bounded by $\left(1+\epsilon\right)\left(p/2\right)\log_{q}\left(n\right)$
and the diameter is bounded by $\left(1+\epsilon\right)p\log_{q}\left(n\right)$.
\item By taking $\tau=(2+\epsilon)\log_q(\log_q(n))$ or $\tau=1+(2+\epsilon)\log_q(\log_q(n))$, we get that for every $\epsilon>0$ and $n$ large enough the almost-diameter is actually bounded by $\log_{q}\left(n\right)+\left(2+\epsilon\right)\log_{q}\left(\log_{q}\left(n\right)\right)$
and the diameter is bounded by $2\log_{q}\left(n\right)+\left(2+\epsilon\right)\log_{q}\left(\log_{q}\left(n\right)\right)$.
\item Both previous remarks can be improved, following the method in \cite{chung1994upper}.
Assume $P\in\R\left[X\right]$ is a polynomial of degree $k$, such
that $P\left(1\right)=1$ and $P\left(\lambda\right)=o\left(\frac{1}{\sqrt{n}}\right)$
for $\left|\lambda\right|\le\lambda_{0}$. Then using the operator
$P\left(A\right)$ instead of the operator $A_{k}$, the same proof
implies that the almost-diameter is bounded by $k$. In \cite{chung1994upper}
it is shown that the optimal choice of $P$ (depending on $\lambda_{0}$)
is some twist of the Chebyshev polynomial of the first kind, which
satisfies for $\left|\lambda\right|\le\lambda_{0}$ that $|P\left(\lambda\right)|\le\cosh\left(k\text{a\ensuremath{\cosh}}\left(\frac{1}{\lambda_{0}}\right)\right)^{-1}$. For Ramanujan graphs, the bound is $|P(\lambda)|\le 2 q^{-k/2}$. It has the effect of replacing $((1+\tau)\log_q(n)+1)^2$ in Theorem~\ref{thm:Diameter-Ramanujan} by $4$, and it thus reduces the almost-diameter of Ramanujan graphs
to $\log_{q}\left(n\right)+O\left(g\left(n\right)\right)$, where
$g\left(n\right)\to\infty$ arbitrarily slowly. In the non-Ramanujan
case this analysis is even better, and improves the coefficient of
$\log_{q}\left(n\right)$ to $\frac{\ln\left(q\right)}{2}\left(\text{a\ensuremath{\cosh}}\left(\frac{1}{\lambda_{0}}\right)\right)^{-1}$. Similar
improvements can be made to the diameter. The main results of our
work cannot be improved similarly, since all the gain from the better
analysis is lost due to the weaker assumptions.
\end{enumerate}
\begin{theorem}
\label{thm:cutoff Ramanujan}For every $\epsilon_{0}>0$, $\epsilon_1>0$ there is $n_0$ such that the following holds. Let $X$ be a non-bipartite graph of size $n> n_0$. Then for every $x_{0}\in X$ the probability distributions
$A^{k}\delta_{x_{0}}$ of the simple random walk on $X$ satisfy the
following:
\begin{enumerate}
\item For $k<\left(1-\epsilon_{0}\right)\frac{q+1}{q-1}\log_{q}\left(n\right)$,
it holds that 
\begin{align*}
\n{A^{k}\delta_{x_{0}}-\pi}_{1}\ge 2-\epsilon_1.
\end{align*}
\item If $X$ is Ramanujan, for $k>\left(1+\epsilon_{0}\right)\frac{q+1}{q-1}\log_{q}\left(n\right)$,
it holds that 
\begin{align*}
\n{A^{k}\delta_{x_{0}}-\pi}_{1}\le \epsilon_1.
\end{align*}
\end{enumerate}
\end{theorem}

The proof of this theorem is based on the following lemma, whose proof
can be found in \cite[Section 2]{lubetzky2016cutoff}.
\begin{lem}
\label{lem:random walk deviation}It holds that $A^{k}=\sum_{i=0}^{k}\alpha_{i}^{(k)}A_{i}$
for some constants $\alpha_{i}^{(k)}$, satisfying $0\le\alpha_{i}^{(k)}$
and $\sum_{i=0}^{k}\alpha_{i}^{(k)}=1$. Moreover, for every $\epsilon_2,\epsilon_3>0$, for $k$ large enough,
\begin{align*}
\sum_{i:\left|i-\frac{q-1}{q+1}k\right|>\epsilon_2 k}\alpha_{i}^{(k)}\le \epsilon_3.
\end{align*}
\end{lem}

The constants $\alpha_{i}^{(k)}$ in the lemma are the probability
that the simple random walk on the $(q+1)$-regular tree starting
from some vertex $x_{0}$ is at distance $i$ from $x_{0}$ after
$k$ steps. Therefore, the lemma is a crude estimate on the rate of escape
of the simple random walk, and it follows from the fact that the random
walk is transient almost-surely, and once we leave for the last time
the root $x_{0}$ we move away from $x_{0}$ with probability $\frac{q}{q+1}$
and move towards $x_{0}$ with probability $\frac{1}{q+1}$. For more
precise statements, including a Central Limit Theorem for this deviation,
see \cite[Section 2]{lubetzky2016cutoff}.

\begin{proof}[Proof of Theorem~\ref{thm:cutoff Ramanujan}]
For $(1)$, assume that $k<\left(1-\epsilon_{0}\right)\frac{q+1}{q-1}\log_{q}\left(n\right)$.
Choose $\epsilon_2>0$ small enough relatively to $\epsilon_{0}$.
It holds, by Lemma~\ref{lem:random walk deviation}, that
\begin{align*}
\n{A^{k}\delta_{x_{0}}-\pi}_{1} & =\n{\sum_{i=0}^{k}\alpha_{i}^{(k)}\left(A_{i}\delta_{x_{0}}-\pi\right)}_{1}\\
 & \ge-\sum_{i>\left(\frac{q-1}{q+1}+\epsilon_{2}\right)k}\sum_{i=0}^{k}\alpha_{i}^{(k)}\n{A_{i}\delta_{x_{0}}-\pi}_{1}+\n{\sum_{i\le\left(\frac{q-1}{q+1}+\epsilon_{2}\right)k}\alpha_{i}^{(k)}\left(A_{i}\delta_{x_{0}}-\pi\right)}_{1}.
\end{align*}
For $n$ large enough, the first term is bounded by  $\epsilon_1 /3$ by Lemma~\ref{lem:random walk deviation}.
In the second term, for $i\le\left(\frac{q-1}{q+1}+\epsilon_{2}\right)k$,
$A_{i}\delta_{x_{0}}$ is supported on a ball of radius at most 
\begin{align*}
\left(\frac{q-1}{q+1}+\epsilon_{2}\right)\left(1-\epsilon_{0}\right)\frac{q+1}{q-1}\log_{q}\left(n\right)\le\left(1-\epsilon_{0}/2\right)\log_{q}\left(n\right)
\end{align*}
around $x_{0}$, where we choose $\epsilon_2$ to be small enough relative to $\epsilon_0$ so that the above will hold. Therefore, $\sum_{i\le\left(\frac{q-1}{q+1}+\epsilon_{2}\right)k}\alpha_{i}^{(k)}A_{i}\delta_{x_0}$ is non-zero on at most 
\begin{align*}
O\left(q^{\left(1-\epsilon_{0}/2\right)\log_{q}\left(n\right)}\right) \ll n^{1-\epsilon_0/2}
\end{align*}
of the vertices of the graph, which implies, using Lemma~\ref{lem:random walk deviation}, 
\begin{align*}
\n{\sum_{i\le\left(\frac{q-1}{q+1}+\epsilon_{2}\right)k}\alpha_{i}^{(k)}\left(A_{i}\delta_{x_{0}}-\pi\right)}_{1} & \ge\left(\sum_{i\le\left(\frac{q-1}{q+1}+\epsilon_{2}\right)k}\alpha_{i}^{(k)}\right)\left(2-O(n^{-\epsilon_0/2})\right)\\
 & \ge(1-\epsilon_1/3)\left(2-O(n^{-\epsilon_0/2})\right).
\end{align*}

Finally, we have for $n$ large enough:
\begin{align*}
\n{A^{k}\delta_{x_{0}}-\pi}_{1} & \ge-\epsilon_1/3+2-\epsilon_1/2\\
 & \ge 2-\epsilon_1.
\end{align*}

For $\left(2\right)$, assume that $k>\left(1+\epsilon_{0}\right)\frac{q+1}{q-1}\log_{q}\left(n\right)$.
Choose $\epsilon_2>0$ small enough relatively to $\epsilon_{0}$.
Then
\begin{align*}
\n{A^{k}\delta_{x_{0}}-\pi}_{1} & =\n{\sum_{i=0}^{k}\alpha_{i}^{(k)}\left(A_{i}\delta_{x_{0}}-\pi\right)}_{1}\\
 & \le\sum_{i\le\left(\frac{q-1}{q+1}-\epsilon_{2}\right)k}\alpha_{i}^{(k)}\n{A_{i}\delta_{x_{0}}-\pi}_{1}+\sum_{i>\left(\frac{q-1}{q+1}-\epsilon_{2}\right)k}\alpha_{i}^{(k)}\n{A_{i}\delta_{x_{0}}-\pi}_{1}.
\end{align*}
The first term is bounded for $n$ large enough by $\epsilon_1/2$ by Lemma~\ref{lem:random walk deviation}.
The second term can be bounded by Cauchy-Schwartz 
\begin{align*}
\sum_{i>\left(\frac{q-1}{q+1}-\epsilon_{2}\right)k}\alpha_{i}^{(k)}\n{\left(A_{i}\delta_{x_{0}}-\pi\right)}_{1} & \le\sup_{i>\left(\frac{q-1}{q+1}-\epsilon_{2}\right)k}\n{A_{i}\delta_{x_{0}}-\pi}_{1}\\
 & \le\sup_{i>\left(\frac{q-1}{q+1}-\epsilon_{2}\right)k}\sqrt{n}\n{A_{i}\delta_{x_{0}}-\pi}_{2}.
\end{align*}
But since $k>\left(1+\epsilon_{0}\right)\frac{q+1}{q-1}\log_{q}\left(n\right)$
and $\epsilon_{2}$ is small enough relatively to $\epsilon_{0}$,
it holds that $i>\left(\frac{q-1}{q+1}-\epsilon_{2}\right)k$
satisfies $i>\left(1+\epsilon_{0}/2\right)\log_{q}\left(n\right)$. Then 
\begin{align*}
\sqrt{n}\n{A_{i}\delta_{x_{0}}-\pi}_{2} & \le\sqrt{n}\n{A_{i}}_{L_{0}^{2}\left(X\right)}\n{\delta_{x_{0}}-\pi}_{2}\\
 & \ll_{\epsilon_3}\sqrt{n}q^{i\left(-1/2+\epsilon_3\right)} \\
 & \ll n^{1/2+(1+\epsilon_0/2)(-1/2+\epsilon_3)}
\end{align*}
By choosing $\epsilon_3>0$ small enough relatively to $\epsilon_{0}$ and $n$ large enough, we get that the last value is also bounded by $\epsilon_1/2$. 
\end{proof}
\begin{rem}
As in Theorem~\ref{thm:Diameter-Ramanujan}, the results can be improved
in the Ramanujan case by a more careful analysis. In particular, the
condition $k<\left(1-\epsilon_{0}\right)\frac{q+1}{q-1}\log_{q}\left(n\right)$
can be replaced by 
\begin{align*}
k<\frac{q+1}{q-1}\log_{q}\left(n\right)-\left(\log_{q}\left(n\right)\right)^{1/2+\epsilon_{0}}
\end{align*}
and similarly for the upper bound (see \cite{lubetzky2016cutoff}
for more details).
\end{rem}

\section{\label{sec:Proof-of-almost-density}Proof of Theorem~\ref{thm:Almost-diameter theorem}
and Theorem~\ref{thm:Cutoff theorem}}

Before proving the theorems we give some other useful definition. 
\begin{defn}
We say that $x_{0}\in X$ has \emph{local Sarnak-Xue property with
parameter $A$ }if for every $k\le2A\log_{q}\left(n\right)$ and $\epsilon>0$
it holds that 
\begin{align*}
P\left(X,k,x_{0}\right)\ll_{\epsilon,\F}n^{\epsilon}q^{k/2}.
\end{align*}
\end{defn}

Let $u_{0},...,u_{n-1}$ be an orthogonal basis of $L^{2}\left(X\right)$
composed of eigenvectors of $A$, with eigenvalue $\lambda_{i}$ and
$p$-value $p_{i}$. 
\begin{lem}
\label{lem:local density partial sum}The vertex $x_{0}\in X$ has
local Sarnak-Xue property with parameter $A$ if and only if for every
$\epsilon>0$, 
\begin{equation}
\sum_{i=0}^{n-1}\left|\left\langle \delta_{x_{0}},u_{i}\right\rangle \right|^{2}n^{A\left(1-2/p_{i}\right)}\ll_{\epsilon,\F}n^{\epsilon}.\label{eq:local density 2}
\end{equation}
\end{lem}

\begin{proof}
It holds that 
\begin{align*}
P\left(X,k,x_{0}\right)=\left(q+1\right)q^{k-1}\left\langle \delta_{x_{0}},A_{k}\delta_{x_{0}}\right\rangle .
\end{align*}
By the spectral decomposition 
\begin{align*}
\delta_{x_{0}}=\sum_{i=0}^{n-1}\left\langle \delta_{x_{0}},u_{i}\right\rangle u_{i}
\end{align*}
we have 
\begin{align*}
P\left(X,k,x_{0}\right) & \asymp q^{k}\sum_{i=0}^{n-1}\left|\left\langle \delta_{x_{0}},u_{i}\right\rangle \right|^{2}\lambda_{i}^{(k)}\\
 & =q^{k}\sum_{i;p_{i}=2}\left|\left\langle \delta_{x_{0}},u_{i}\right\rangle \right|^{2}\lambda_{i}^{(k)}+q^{k}\sum_{i;p_{i}>2}\left|\left\langle \delta_{x_{0}},u_{i}\right\rangle \right|^{2}\lambda_{i}^{(k)}.
\end{align*}
the first term always satisfies for every $\epsilon>0$,
\begin{align*}
q^{k}\left|\sum_{i;p_{i}=2}\left|\left\langle \delta_{x_{0}},u_{i}\right\rangle \right|^{2}\lambda_{i}^{\left(k\right)}\right| & \ll_{\epsilon}q^{k}q^{-k/2}q^{k\epsilon}\sum_{i;p_{i}=2}\left|\left\langle \delta_{x_{0}},u_{i}\right\rangle \right|^{2}\\
 & \le q^{k\epsilon}q^{k/2}.
\end{align*}
Therefore, the local Sarnak-Xue density is satisfied if and only if
for every $k\le2A\log_{q}\left(n\right)$ and $\epsilon>0$,
\begin{align*}
q^{k}\sum_{i;p_{i}>2}\left|\left\langle \delta_{x_{0}},u_{i}\right\rangle \right|^{2}\lambda_{i}^{(k)}\ll_{\epsilon,\F} n^{\epsilon}q^{k/2}.
\end{align*}
Since for $k=0\mod2$ it holds that $q^{k/p_{i}}\le\lambda_{i}^{(k)}\ll_{\epsilon}q^{k\left(1/p_{i}+\epsilon\right)}$
and for $k=1\mod2$ we know the upper bound in absolute value, the
last condition holds if and only if for every $k\le2A\log_{q}\left(n\right)$ and $\epsilon>0$,
\begin{align*}
q^{k}\sum_{i;p_{i}>2}\left|\left\langle \delta_{x_{0}},u_{i}\right\rangle \right|^{2}q^{-k\left(1/p_{i}+\epsilon\right)}\ll_{\epsilon,\F} n^{\epsilon}q^{k/2}
\end{align*}
i.e., the local Sarnak-Xue density with parameter $A$ is satisfied
if and only if for every $k\le2A\log_{q}\left(n\right)$ and $\epsilon>0$,
\begin{align*}
\sum_{i;p_{i}>2}\left|\left\langle \delta_{x_{0}},u_{i}\right\rangle \right|^{2}q^{k\left(1/2-1/p_{i}\right)}\ll_{\epsilon,\F}n^{\epsilon}.
\end{align*}
This condition for $k=2A\left\lfloor \log_{q}\left(n\right)\right\rfloor $
shows that local Sarnak-Xue density with parameter $A$ implies Equation~\eqref{eq:local density 2}.
But if this condition holds for $k$ then it obviously holds for $k'\le k$.
Therefore, Equation~\eqref{eq:local density 2} also implies the local
Sarnak-Xue density with parameter $A$.
\end{proof}

\begin{lem}
\label{lem:almost-radius from density}If a vertex $x_{0}\in X$ has
local Sarnak-Xue density with parameter $A=1$ and $X$ is an expander
then:
\begin{itemize}
\item For every $\epsilon_{0},\epsilon_1>0$, for $n$ large enough depending on $\F,\epsilon_0,\epsilon_1$, for all but $\epsilon_0 n $
of $y\in X$ it holds that $d\left(x,y\right)\le\left(1+\epsilon_{1}\right)\log_{q}\left(n\right)$,
i.e. $R\left(n\right)=\left(1+\epsilon_{1}\right)\log_{q}\left(n\right)$
is an almost-radius of $X$ at $x_{0}$.
\item For every $\epsilon_{0},\epsilon_1>0$, for $n$ large enough depending on $\F,\epsilon_0,\epsilon_1$, for $k>\left(1+\epsilon_{1}\right)\frac{q+1}{q-1}\log_{q}\left(n\right)$,
it holds that 
\begin{align*}
\n{A^{k}\delta_{x_{0}}-\pi}_{1}\le \epsilon_0.
\end{align*}
\end{itemize}
\end{lem}

\begin{proof}
As before, we assume that $X$ is non-bipartite, and the bipartite
case can be treated similarly. We let $u_{0},...,u_{n-1}$ be the
orthogonal basis of $L^{2}\left(X\right)$ as above. We assume that
$u_{0}=\frac{\pi}{\sqrt{n}}$ the $L^{2}$-normalized constant function. 

The fact that $X$ is an expander means that there exists $p'<\infty$
(depending only on $\mathcal{F}$) such that for all $i>0$ it holds
that $p_{i}\le p'$. 

Let $k\ge\left\lfloor \left(1+\epsilon_{1}\right)\log_{q}\left(n\right)\right\rfloor $.
By Lemma~\ref{lem:L2 to diameter lemma} and the proof of Theorem~\ref{thm:cutoff Ramanujan}, to prove both claims in the theorem it suffices to prove that 
\begin{align*}
\n{A_{k}\left(\delta_{x_{0}}-\pi\right)}_{2}^{2}\le \epsilon_0 n^{-1}.
\end{align*}

Decompose $\delta_{x_{0}}-\pi=\sum_{i>1}\left\langle \delta_{x_{0}},u_{i}\right\rangle u_{i}$.
Then by Corollary \ref{cor:eigenvalue calculations}, for every $\epsilon_2>0$.
\begin{align*}
\n{A_{k}\left(\delta_{x_{0}}-\pi\right)}_{2}^{2} & =\sum_{i=1}^{n-1}\left|\left\langle \delta_{x_{0}},u_{i}\right\rangle \right|^{2}\left|\lambda_{i}^{\left(k\right)}\right|^{2}\\
 & \ll_{\epsilon_2}\sum_{i=1}^{n-1}\left|\left\langle \delta_{x_{0}},u_{i}\right\rangle \right|^{2}q^{2k\left(-1/p_{i}+\epsilon_2\right)}\\
 & \ll\sum_{i=1}^{n-1}\left|\left\langle \delta_{x_{0}},u_{i}\right\rangle \right|^{2}n^{2\left(-1/p_{i}+\epsilon_2\right)\left(1+\epsilon_{1}\right)}\\
 & =\sum_{i=1}^{n-1}\left|\left\langle \delta_{x_{0}},u_{i}\right\rangle \right|^{2}n^{-2/p_{i}}n^{2\left(\left(-1/p_i+\epsilon_2\right)\epsilon_{1}+\epsilon_2\right)} \\
 & \le n^{2\left(\left(-1/p'+\epsilon_2\right)\epsilon_{1}+\epsilon_2\right)} \sum_{i=1}^{n-1}\left|\left\langle \delta_{x_{0}},u_{i}\right\rangle \right|^{2}n^{-2/p_{i}}.
\end{align*}
If we choose $\epsilon_2>0$ small enough relative to $\epsilon_1,p'$ it holds that for every $\epsilon_3>0$,
\begin{align*}
 & \ll n^{-\epsilon_{1}/p'}\sum_{i=2}^{n}\left|\left\langle \delta_{x_{0}},u_{i}\right\rangle \right|^{2}n^{-2/p_{i}}\\
 & \ll_{\epsilon_{3},\F}n^{-\epsilon_{1}/p'}n^{-1+\epsilon_{3}},
\end{align*}
where we used Lemma~\ref{lem:local density partial sum} and the local Sarnak-Xue condition. Choose a small enough constant $c_{\F}$, and then choose $\epsilon_{3}>0$ small enough and $n$ large enough relative to $\epsilon_1,p'$, so that the last value is $\le c_{\F}\epsilon_0 n^{-1}$. 

By adjusting $c_{\F}$, we get eventually that 
\begin{align*}
\n{A_{k}\left(\delta_{x_{0}}-\pi\right)}_{2}^{2}\le \epsilon_0 n^{-1},
\end{align*}
as needed.
\end{proof}
\begin{lem}
\label{lem:global to local density}Assume that $\mathcal{F}$ is
a family of graphs satisfying the Sarnak-Xue density with parameter $A=1$. Let $\epsilon_0>0$. Then we may choose for each graph $X\in\mathcal{F}$ a subset
$Y\subset X$ of the vertices with $\left|Y\right|\ge\left|X\right|\left(1-\epsilon_0\right)$
such that for every $x_{0}\in Y$ the local Sarnak-Xue property with
parameter $A=1$ is satisfied at $x_{0}$.
\end{lem}

\begin{proof}
Since $P\left(X,k\right)=\sum P\left(X,k,x_{0}\right)$, if the Sarnak-Xue
density property (or equivalently,t he weak injective radius property)
with parameter $A=1$ holds then the local Sarnak-Xue density holds
\emph{on average} over all the vertices $x_{0}$. More precisely,
for every $\epsilon_{1}>0$, the number of $x_{0}\in X$
satisfying $P\left(X,2\left\lfloor \log_{q}\left(n\right)\right\rfloor ,x_{0}\right)>n^{1+2\epsilon_{1}}$
is at most $C_{\epsilon_{1}}n^{1-\epsilon_{1}}$, $C_{\epsilon_{1}}$
some constant.

Now, for each $k=1,2,...$ let $\epsilon_{k}=1/k$. Let $N_{k}$ be
large enough so that $C_{\epsilon_{k}}N_{k}^{1-\epsilon_{k}/2} \le \epsilon_0$.
Now for each $X$ choose $k$ maximal such that $n=\left|X\right|>N_{k}$.
Let $Y\subset X$ be the set of vertices $x_{0}\in X$ such that $P\left(X,2\left\lfloor \log_{q}\left(n\right)\right\rfloor ,x_{0}\right)\le n^{1+\epsilon_{k}}$.
By construction $\left|Y\right|\ge\left|X\right|\left(1-\epsilon_0\right)$.

We show that the local Sarnak-Xue density holds: for every $\epsilon>0$
let $k$ be such that $\epsilon_{k}<2\epsilon$. Then for $\left|X\right|>N_{k}$,
for $x_{0}\in Y$, $P\left(X,2\left\lfloor \log_{q}\left(n\right)\right\rfloor ,x_{0}\right)\le n^{1+\epsilon_{k}/2}\le n^{1+\epsilon}$.
Since there is a finite number of graphs $X\in\mathcal{F}$ with $\left|X\right|<N_{k}$
we are done.
\end{proof}
We can now prove Theorem~\ref{thm:Almost-diameter theorem} and Theorem~\ref{thm:Cutoff theorem}.

\begin{proof}[Proof of Theorem~\ref{thm:Almost-diameter theorem} and Theorem~\ref{thm:Cutoff theorem}]
The proof follows from the combination of Lemma~\ref{lem:global to local density}
and Lemma~\ref{lem:almost-radius from density}.
\end{proof}

As a final claim, we prove the following lemma, needed in the proof of Theorem~\ref{thm:Projective space theorem}.
\begin{lem}
\label{lem:Cayley paths}If $X$ is a Cayley graph which satisfies
the Sarnak-Xue density property with parameter $A$,
then for every $x_{0},y_{0}\in X$ and $k<2A\log_{q}n$, we have for
every $\epsilon>0$ 
\begin{align*}
P\left(X,k,x_{0},y_{0}\right)\ll_{\epsilon,\mathcal{F}}n^{\epsilon}q^{k/2},
\end{align*}
where $P\left(X,k,x_{0},y_{0}\right)$ is the number of non-backtracking
paths from $x_{0}$ to $y_{0}$.
\end{lem}

\begin{proof}
It holds that 
\begin{align*}
P\left(X,k,x_{0},y_{0}\right)\le\left(q+1\right)^{2}q^{k-2}\left\langle A_{\left\lfloor k/2\right\rfloor }\delta_{x_{0}},A_{\left\lceil k/2\right\rceil }\delta_{x_{0}}\right\rangle ,
\end{align*}
since the right hand side counts more paths than just non-backtracking
paths between $x_{0}$ and $y_{0}$. Therefore, by the Cauchy-Schwartz
inequality,
\begin{equation}
P\left(X,k,x_{0},y_{0}\right)\ll q^{k}\n{A_{\left\lfloor k/2\right\rfloor }\delta_{x_{0}}}_{2}\n{A_{\left\lceil k/2\right\rceil }\delta_{x_{0}}}_{2}.\label{eq:paths proofs}
\end{equation}
Since $X$ is a Cayley graph, both $x_{0}$ and $y_{0}$ have local
Sarnak-Xue property with parameter $A$, which means that\emph{ }if
for every $k\le2A\log_{q}\left(n\right)$ and $\epsilon>0$ it holds
that 
\begin{align*}
P\left(X,k,x_{0}\right)=P\left(X,k,y_{0}\right)\ll_{\mathcal{F},\epsilon}n^{\epsilon}q^{k/2}.
\end{align*}
As in Lemma~\ref{lem:local density partial sum}, this is equivalent
to the fact that for every $\epsilon>0$
\begin{align*}
\sum_{i=0}^{n-1}\left|\left\langle \delta_{x_{0}},u_{i}\right\rangle \right|^{2}n^{A\left(1-2/p_{i}\right)}\ll_{\mathcal{F},\epsilon}n^{\epsilon}.
\end{align*}
Therefore, for $k\le2A\log_{q}\left(n\right)$ and $\epsilon>0$,
\begin{align*}
\n{A_{\left\lfloor k/2\right\rfloor }\delta_{x_{0}}}_{2}^{2} & =\sum_{i=0}^{n-1}\left|\lambda_{i}^{(\left\lfloor k/2\right\rfloor )}\right|^{2}\left|\left\langle \delta_{x_{0}},u_{i}\right\rangle \right|^{2}\\
 & \ll_{\epsilon}\sum_{i=0}^{n-1}q^{-k/p_{i}}\left|\left\langle \delta_{x_{0}},u_{i}\right\rangle \right|^{2}\\ 
 & =q^{-k/2}\sum_{i=0}^{n-1}q^{k/2-k/p_{i}}\left|\left\langle \delta_{x_{0}},u_{i}\right\rangle \right|^{2}\\
 & \ll q^{-k/2}\sum_{i=0}^{n-1}q^{A\left(1-2/p_{i}\right)}\left|\left\langle \delta_{x_{0}},u_{i}\right\rangle \right|^{2}\\
 & \ll_{\epsilon,\F}q^{-k/2}n^{\epsilon}.
\end{align*}

Inserting this and the similar claim for $\left\lceil k/2\right\rceil $
into Equation~\eqref{eq:paths proofs} we have for every $\epsilon>0$,
\begin{align*}
P\left(X,k,x_{0},y_{0}\right)\ll_{\epsilon,\F}n^{\epsilon}q^{k/2},
\end{align*}
as needed.
\end{proof}

\section{\label{sec:Benjamini-Schramm}Benjamini-Schramm Convergence}

Our goal in this section is to state the following theorem, which
follows from \cite{abert2014kesten} and \cite{abert2016measurable}.
\begin{theorem}
Let $\F$ be a family of $\left(q+1\right)$-regular graphs with the
number of vertices growing to infinity. The following are equivalent,
and if they hold we say that the sequence of graphs Benjamini-Schramm
converges to the $\left(q+1\right)$-regular tree:
\end{theorem}

\begin{enumerate}
\item For every $k>0$, as $n\to\infty$
\begin{align*}
\lim_{n\to\infty}P\left(X,k\right)/n\to0.
\end{align*}
\item For every $k>0$
\begin{align*}
P\left(X,k\right)/n\ll_{\epsilon,k}n^{\epsilon}q^{k/2}.
\end{align*}
\item For every $\epsilon>0$,
\begin{align*}
\#\left\{ \lambda\in\operatorname{spec}A:\left|\lambda\right|>\left(1+\epsilon\right)2\sqrt{q}\right\} /n\to0.
\end{align*}
\item The spectral measure of $A$ converges to the spectral measure of
the tree.
\end{enumerate}
\begin{proof}
The equivalence of $(1)$,$(3)$ and $(4)$ is a consequence of \cite[Theorem 5]{abert2016measurable}
(the result is slightly more general than \cite[Theorem 4]{abert2016measurable},
but follows in the same way). 

The fact that $(2)$ and $(3)$ are equivalent is proven in the same
way as Theorem~\ref{thm:SX-equivalence}, and is a slightly stronger
version of \cite[Corollary 7]{abert2016measurable}.
\end{proof}
One interesting corollary of the proof of Theorem~\ref{thm:Projective space theorem}
for the context of Benjamini-Schramm convergence is the following:
\begin{cor}
If we have a family $\text{Cayley}\left(\SL_{2}\left(\mathbb{F}_{t}\right),S_{t}\right)$,
where $t$ is prime, which Benjamini-Schramm converges to the $(q+1)$-regular
tree, then the corresponding family of Schreier graphs on $P^{1}\left(\mathbb{F}_{t}\right)$
also Benjamini-Schramm converges to the $(q+1)$-regular tree.
\end{cor}

\bibliographystyle{amsplain}
\bibliography{./database}

\end{document}